\documentclass[a4paper,10pt]{amsart}
 \usepackage{amsfonts,mathrsfs}
 \usepackage{amsthm}
 \usepackage{amssymb}
 \usepackage{enumerate}
 \usepackage{tikz-cd}
 \usepackage{comment}
 \usepackage{hyperref}
 \usepackage{cleveref}

\theoremstyle{definition}
\newtheorem{Def}{Definition}[section]
\newtheorem*{Def*}{Definition}
\newtheorem{ex}[Def]{Example}
\newtheorem{cons}[Def]{Construction}
\newtheorem{rem}[Def]{Remark}

\theoremstyle{plain}
\newtheorem{prop}[Def]{Proposition}
\newtheorem{thm}[Def]{Theorem}
\newtheorem*{thm*}{Theorem}
\newtheorem{lem}[Def]{Lemma}
\newtheorem{cor}[Def]{Corollary}
\newtheorem*{cor*}{Corollary}

\newtheorem*{con*}{Conjecture}
\newtheorem{qu}{Question}
\newtheorem*{verm*}{Vermutung}

\newcommand{\thistheoremname}{}
\newtheorem*{genericthm*}{\thistheoremname}
\newenvironment{namedthm*}[1]
{\renewcommand{\thistheoremname}{#1}
	\begin{genericthm*}}
	{\end{genericthm*}}
\newenvironment{namedtheorem*}[1]
{\renewcommand{\thistheoremname}{#1}
	\begin{genericthm*}}
	{\end{genericthm*}}

\usepackage{mathrsfs}

\newcommand{\Hank}{\operatorname{H}}
\newcommand{\B}{\operatorname{B}}
\newcommand{\naiv}{\operatorname{N}}
\newcommand{\pr}{\operatorname{pr}}

\newcommand{\iso}{\operatorname{I}}

\newcommand{\Pic}{\operatorname{Pic}}
\newcommand{\Mat}{\operatorname{Mat}}

\newcommand{\Hom}{\operatorname{Hom}} 
\newcommand{\shHom}{\operatorname{{\mathscr Hom}}} 
\newcommand{\Spec}{\operatorname{Spec}}

\newcommand{\rank}{\operatorname{rank}} 

\newcommand{\codim}{\operatorname{codim}}

\newcommand{\depth}{\operatorname{depth}}

\newcommand{\Ext}{\operatorname{Ext}} 
\newcommand{\shExt}{\operatorname{{\mathscr Ext}}} 

\newcommand{\T}{\operatorname{T}}
\newcommand{\J}{\operatorname{J}}
\newcommand{\V}{\operatorname{V}}

\newcommand{\divv}{\operatorname{div}}

\newcommand{\Char}{\operatorname{char}}

\newcommand{\w}{\operatorname{w}}
\newcommand{\kosz}{\operatorname{\bf K}}
\newcommand{\Ch}{\mathscr{C}}

\newcommand{\GL}{\operatorname{{\mathbf GL}}}

\newcommand{\Quot}{\operatorname{Quot}}
\newcommand{\SL}{\operatorname{{\mathbf SL}}}

\newcommand{\tr}{\operatorname{tr}}
\newcommand{\Tr}{\operatorname{Tr}}

\newcommand{\GW}{\operatorname{GW}}

\newcommand{\cB}{{\mathcal B}}

\newcommand{\cF}{{\mathcal F}}
\newcommand{\cG}{{\mathcal G}}

\newcommand{\cL}{{\mathcal L}}

\newcommand{\cO}{{\mathcal O}}

\newcommand{\cT}{{\mathcal T}}

\newcommand{\A}{{\mathbb A}}
\newcommand{\G}{{\mathbb G}}
\newcommand{\C}{{\mathbb C}}
\newcommand{\R}{{\mathbb R}}
\newcommand{\pp}{\mathbb{P}}
\newcommand{\Q}{{\mathbb Q}}
\newcommand{\N}{{\mathbb N}}
\newcommand{\Z}{{\mathbb Z}}

\title[Ulrich sheaves, the writhe and algebraic isotopies]{Ulrich sheaves, the arithmetic writhe and algebraic isotopies of space curves}
\author{Daniele Agostini}
\address{Eberhard Karls Universit\"at T\"ubingen, T\"ubingen, Germany } 
\email{daniele.agostini@uni-tuebingen.de}
\author{Mario Kummer}
\address{Technische Universit\"at, Dresden, Germany} 
\email{mario.kummer@tu-dresden.de}
\thanks{Daniele Agostini has been supported by the Deutsche Forchungsgemeinschaft under Grant No. 530132094 and by the SFB-TRR 195.
Mario Kummer has been supported by the Deutsche Forschungsgemeinschaft under Grant No. 502861109.}
\begin{document}


\begin{abstract}
 We establish a connection between the theory of Ulrich sheaves and $\A^1$-homotopy theory. For instance, we prove that the $\A^1$-degree of a morphism between projective varieties, that is relatively oriented by an Ulrich sheaf, is constant on the target even when it is not $\A^1$-chain connected or $\A^1$-connected. Further if an embedded projective variety is the support of a symmetric Ulrich sheaf of rank one, the $\A^1$-degree of all its linear projections can be read off in an explicit way from the free resolution of the Ulrich sheaf. Finally, we construct an Ulrich sheaf on the secant variety of a curve and use this to define an arithmetic version of Viro's encomplexed writhe for curves in $\pp^3$. This can be considered to be an arithmetic analogue of a knot invariant. Namely, we define a notion of algebraic isotopy under which the arithmetic writhe is invariant. For rational curves of degree at most four in $\pp^3$ we obtain a complete classification up to algebraic isotopies.
\end{abstract}
\maketitle

\section{Introduction}
In the emerging field of $\A^1$-enumerative geometry \cite{kw19,Lev20,kass}, building upon the $\A^1$-homotopy theory developed by Morel and Voevodsky \cite{MV99}, solutions to enumerative problems over a field $K$ are counted in the Grothendieck--Witt group $\GW(K)$ of this field in a way that the result does not depend on the chosen instance of the enumerative problem. Typical examples of enumerative problems whose solutions can be ``arithmetically enriched'' in this way include among others the 27 lines on a cubic surface \cite{kass}, B\'ezout's theorem \cite{mckean} or Gromov--Witten invariants \cite{KLSW}. Many results of this type are achieved by studying arithmetic versions of topological invariants like the Euler characteristic \cite{levineeuler} or the Brouwer degree \cite{morel,kw19}. Here we focus on the latter, namely the \emph{$\A^1$-degree} of a morphism of smooth schemes as considered in \cite{KLSW,pauliwickelgren}. This is defined for a finite and surjective morphism $f\colon X\to Y$ of smooth $K$-varieties starting from a \emph{relative orientation} of $f$, this is a line bundle $L$ on $X$, with an isomorphism
\[ \psi\colon L\otimes L \longrightarrow \shHom(\det\cT_X, f^*\det\cT_Y). \]
With this, we can compute the $\mathbb{A}^1$-degree at a closed point $y\in Y$ as a sum of local degrees on the fiber, as in the classical case; see Section \ref{sec:orientations} for more details. When this is independent of the point $y\in Y$, one says that the $\A^1$-degree of the map $f$ is \emph{well-defined}. This is not always the case. One important case, when this happens, is if $Y$ is $\A^1$-chain connected. Our first result is a different condition for the well-definedness of the $\A^1$-degree in terms of \emph{Ulrich sheaves}: a sheaf $\mathcal{F}$ on $X$ is called \emph{$f$-Ulrich} if its push-forward along $f$ trivial, i.e., $f_*\mathcal{F} \cong \mathcal{O}_{Y}^{\oplus N}$ for $N\in \mathbb{N}$.

\begin{namedthm*}{Theorem A}
If the relative orientation $L$ is $f$-Ulrich and if $H^0(Y,\mathcal{O}_Y)=K$ then the $\mathbb{A}^1$-degree of $f$ is well-defined.
\end{namedthm*}

We refer to \Cref{thm:ulrichdegreewell} for a precise statement of the result and for the proof. In particular, we consider the case of standard Ulrich sheaves on an embedded irreducible projective variety $X\subset \mathbb{P}^n$ of dimension $\dim X=k$: these are sheaves $\mathcal{F}$ on $X$ which are $\pi$-Ulrich for any finite linear projection $\pi\colon X\to \pp^k$. Ulrich sheaves were introduced by Eisenbud and Schreyer in \cite{ES03} and they satisfy a plethora of nice properties. In particular, there has been ample interest in the question whether every closed subvariety of $\pp^n$ carries an Ulrich sheaf \cite{ES03,beauville}. 

One of the main features of Ulrich sheaves is to give a determinantal representation of the \emph{Chow form} of $X$. Recall that a finite linear projection $\pi\colon X\to \pp^k$ has the form $\pi = [s_0,\dots,s_k]$, where $s_0,\dots,s_k \in H^0(\pp^n,\mathcal{O}_{\pp^n}(1))$ are linearly independent, and such that the linear space $\Pi=\{s_0 = \dots = s_k = 0\} \subset \pp^n$ does not intersect $X$. The $s_0,\dots,s_k$ satisfy these conditions if and only if $\mathscr{C}_X(s_0\wedge \dots \wedge s_k) \ne 0$, for a homogeneous form on $\wedge^{k+1}H^0(\pp^n,\mathcal{O}_{\pp^n}(1))$ called the Chow form of the embedded variety $X\subset \pp^n$. One of the main results of \cite{ES03} is that if $X$ carries an Ulrich sheaf of rank one, then there is a matrix $\Lambda$ of linear forms on $\wedge^{k+1}H^0(\pp^n,\mathcal{O}_{\pp^n}(1))$ whose determinant gives the Chow form. Furthermore, there is a natural notion of \emph{symmetry} for the Ulrich sheaf which guarantees that $\Lambda$ can be taken to be symmetric. This matrix encodes the $\A^1$-degrees of all finite linear projections of $X$:

\begin{namedthm*}{Theorem B}
Let $\mathcal{F}$ be a symmetric Ulrich sheaf of rank one on $X\subset \pp^n$. Then this sheaf induces a relative orientation on each finite linear projection 
\[ [s_0,\dots,s_k]\colon X\to \pp^k \]
such that the resulting $\A^1$-degree is well-defined. Furthermore, there is a symmetric matrix $\Lambda$ of linear forms on $\wedge^{k+1}H^0(\pp^n,\mathcal{O}_{\pp^n}(1))$ whose determinant is the Chow form of $X$, and such that the class 
\[ [\Lambda(s_0\wedge \dots \wedge s_k)] \in \GW(K) \]
is precisely the $\A^1$-degree of the linear projection $[s_0,\dots,s_k]\colon X\to \pp^k$, with respect to the above orientation.
\end{namedthm*}

We refer to \Cref{thm:a1chowmatrix} for a precise statement. We also note that the symmetric matrix $\Lambda$ is explicitly computable from a resolution of the sheaf $\mathcal{F}$, for example via a computer algebra system such as Macaulay2 \cite{M2} or OSCAR \cite{OSCAR}, as already demonstrated in \cite{ES03}.
\vspace{10pt}

\begin{figure}[h]
 \includegraphics[width=4cm]{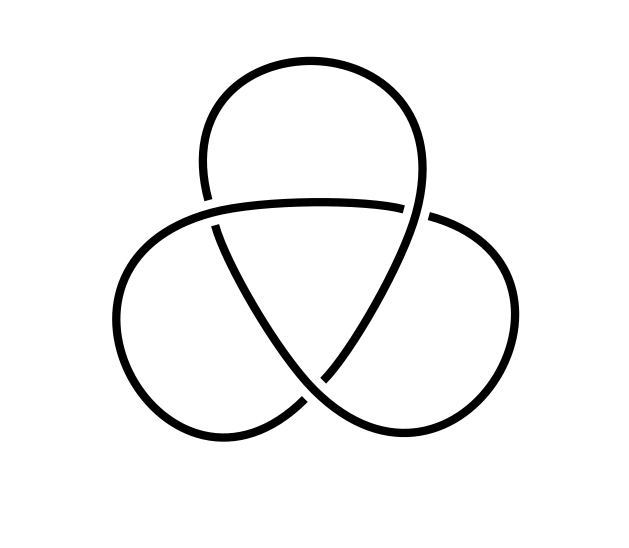}
 \caption{A knot diagram of the trefoil knot.}\label{fig:trefoil}
\end{figure}
With these results in hand, we pass to the second main theme of the paper, an $\A^1$-analogue of knot theory. Classical knots can be studied via \emph{knot diagrams}. These are planar projections of the knot, which, for each point where two arcs cross, keeps track of the over- and under-passing arc, see e.g. \Cref{fig:trefoil}.
The \emph{local writhe} of a crossing in a knot diagram is a number in $\{-1,+1\}$ and is defined as in \Cref{fig:locwrithe}. 
\begin{figure}
 \includegraphics[width=2cm]{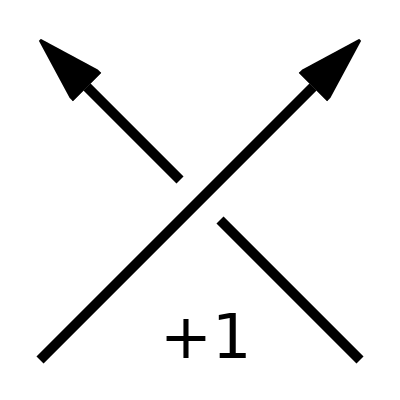} \quad
 \includegraphics[width=2cm]{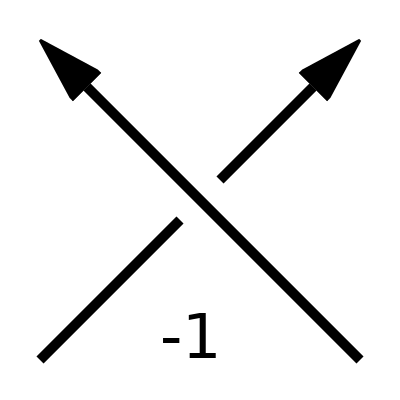}
 \caption{Local writhe numbers.}\label{fig:locwrithe}
\end{figure}
Note that one has to choose an orientation of the knot but the local writhe numbers are independent of this choice.
The \emph{writhe} of a knot diagram is then the sum of the local writhe numbers of all its crossings. For example the writhe of the knot diagram in \Cref{fig:trefoil} is equal to $+3$.
\begin{figure}
    \centering
    \includegraphics[width=2cm]{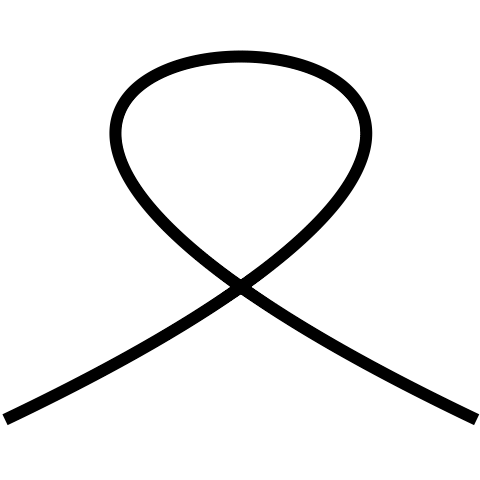}\hspace{1cm}
\includegraphics[width=2cm]{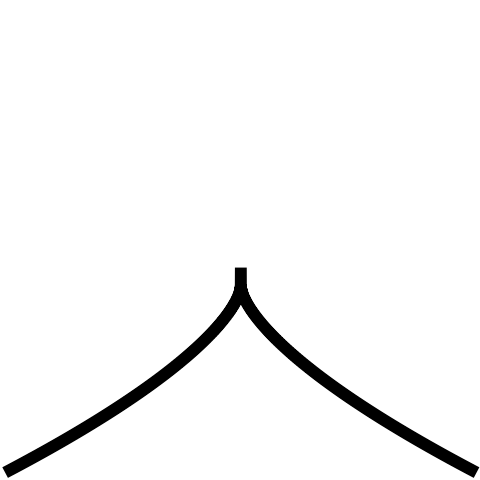}\hspace{1cm}
\includegraphics[width=2cm]{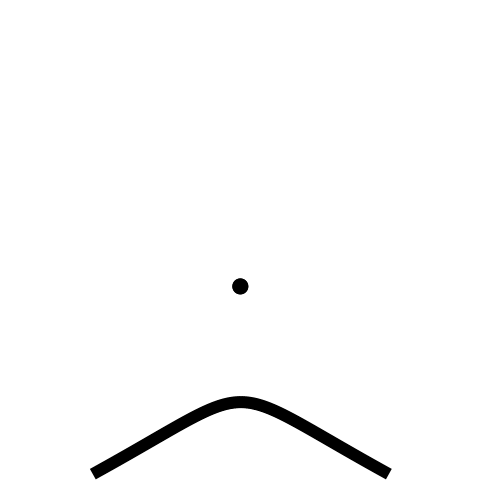}
    \caption{When applying a Reidemeister move of type I to a real algebraic curve, a crossing of two arcs becomes an isolated node.}
    \label{fig:solnode}
\end{figure}
While the writhe is unaffected by Reidemeister moves of type II and III, a Reidemeister move of type I changes the writhe by $\pm 1$. In particular, the writhe is only an invariant of the knot diagram but not of the knot itself. When the knot is the real part of an algebraic curve $C \subset \pp^3$, this issue corresponds to the fact that some plane projections might have isolated real nodes, see \Cref{fig:solnode}.
Viro \cite{virowrithe} showed that the problem can be circumvented by also assigning a local writhe to \emph{isolated} real nodes of a planar projection of the algebraic curve: he defined the \emph{encomplexed writhe number} as the sum over all local writhe numbers, including those at isolated nodes, and proved that this is indeed independent of the projection center, hence it is an invariant of the real curve $C\subset \pp^3$. This invariant played a major role in the recent breakthrough on real algebraic links by Mikhalkin and Orevkov \cite{maxwrithed}. 

We generalize Viro's encomplexed writhe to any field $K$ via Ulrich bundles on secant varieties. 
This is natural, since the nodes of a linear projection $\pi\colon C \to \pp^2$ corresponds to the secant lines to $C$ passing through the center of the projection. Let us be more precise here. If $\psi\colon C\hookrightarrow \pp^3$ is a non-degenerate embedding and $L = \psi^*\mathcal{O}_C(1)$ the corresponding line bundle, then the embedding $\psi$ is obtained from composing the embedding of the complete linear system $\varphi_L\colon C\hookrightarrow \pp(H^0(C,L)^{\vee}) = \pp^n$ with a linear projection $[s_0,s_1,s_2,s_3]\colon \varphi_L(C) \to \pp^3$. In order for $\psi$ to be an embedding, the sections  $s_0,\dots,s_3\in H^0(C,L)$ must be linearly independent, and the space $\Pi = \{s_0=\dots=s_3\}$ should not intersect the secant variety $\Sigma=\Sigma(C,L)$ of $\varphi_L(C)$ in $\pp^n$. In this setting, the induced projection $[s_0,\dots,s_3]\colon \Sigma \to \pp^3$ is a finite and surjective map. Assume now that the secant variety $\Sigma$ is \emph{identifiable} in the sense that every point in $\Sigma\setminus \varphi_L(C)$ is contained in a unique secant line. Then, by construction, the fiber of $[s_0,\dots,s_3]\colon \Sigma \to \pp^3$ over a point $q\in \pp^3$ correspond to secant lines to $\psi(C)$ passing through $q$.
Thus, if we can find a suitable orientation, the $\A^1$-degree of this map gives an arithmetic count of secant lines through $q$, or, equivalently, an arithmetic count of nodes of the projection of $\psi(C)$ from $q$. We can actually find a suitable orientation via an Ulrich sheaf:

\begin{namedthm*}{Theorem C}
Let $C$ be a smooth $K$-curve and $L$ be a line bundle such that the secant variety $\Sigma$ of $\varphi_L\colon C\hookrightarrow \pp^n$ is identifiable. Let also $\alpha$ be a non-effective theta characteristic on $C$, i.e., a line bundle such that $\alpha\otimes \alpha\cong \omega_C$ and $h^0(C,\alpha)=0$. 
Then there is a symmetric Ulrich sheaf of rank one $\mathcal{F}_{\alpha}$ on $\Sigma$, depending on $\alpha$.
\end{namedthm*}

With such a symmetric Ulrich sheaf we can then use Theorem B, and give the following definition:

\begin{Def*}[Arithmetic writhe]
In the above setting, the \emph{arithmetic writhe} 
\[ \w(\psi(C),\alpha) \in \GW(K)\] 
of the embedded curve $\psi\colon C\hookrightarrow \pp^3$ and of the theta characteristic $\alpha$, is the $\A^1$-degree of the map 
\[ [s_0,\dots,s_3]\colon  \Sigma \to \pp^3,\] with respect to the relative orientation induced by $\mathcal{F}_{\alpha}$. Here $[s_0,\dots,s_3]$ is such that $\psi = [s_0,\dots,s_3]\circ \varphi_L$.
\end{Def*}

By the previous discussion, the arithmetic writhe is  an arithmetically enriched count of the nodes of plane projections $\psi(C)\to \pp^2$. The fact that the arithmetic writhe is well-defined as an $\A^1$-degree, means precisely that this does not depend on the center of the  plane projection. We also provide an explicit local description of the arithmetic writhe which for real curves agrees with the local description of Viro's encomplexed writhe. Hence, for real curves, the arithmetic writhe is equal to the encomplexed writhe.

Observe that the appearance of a non-effective theta characteristic $\alpha$ on $C$ is not surprising: indeed these correspond on the one hand to symmetric rank one Ulrich sheaves on the curve itself (see \cite[Theorem 4.3]{ES03}), and, on the other hand, if the curve is real, to classical orientations, see \Cref{ex:thetachars}. Over an arbitrary field, every smooth rational curve has a  non-effective theta characteristic. We also remark that the condition of identifiability on the secant variety of $\varphi_L\colon C\hookrightarrow \pp^n$ is rather natural and general. For example, Riemann--Roch shows that this happens as soon as the degree $\deg(C)$ is at least $2g(C)+3$, where $g(C)$ is the genus of the curve. In particular, this is true for any smooth rational curve of degree at least $3$. Another important case are canonical curves of gonality at least $5$, such as any general real or complex curve in the moduli space $\mathcal{M}_g, g\geq 7$. 
We note that, in a similar spirit, Lemari\'e--Rieusset \cite{clementine} has defined an arithmetic linking degree for two embeddings of $\A^2\setminus\{0\}$ to $\A^4\setminus\{0\}$ using motivic Seifert classes. On the other hand, to our knowledge, the arithmetic writhe is the first \emph{knot} invariant defined in the $\A^1$-context. 
\vspace{10pt}

In the last part of the paper, we define and study \emph{algebraic isotopies} between embeddings of a smooth projective curve $C$ in $3$-space in analogy with topological isotopies: They are morphisms 
\[ I\colon \A^1\times C \to \pp^3, \qquad (t,x) \mapsto I_t(x) \]
such that $I_{t}\colon C_{\kappa(t)} \hookrightarrow \pp^3_{\kappa(t)}$ is an embedding for all $t\in \A^1$. Two embeddings $\psi,\psi'\colon C\hookrightarrow \pp^3$ are algebraically isotopic if they can be connected by a chain of algebraic isotopies. We extend our definition of the arithmetic writhe $\w(\psi(C),\alpha)$ for an embedded curve $\psi\colon C \hookrightarrow \pp^3$ to a definition of the \emph{arithmetic writhe} $\w(\psi,\alpha)$ for the embedding $\psi$ itself, and we show that this is invariant under algebraic isotopies, see \Cref{rem:writheemb}. We use it to study algebraic isotopy classes of rational curves of low degree in $\pp^3$.

\begin{namedthm*}{Theorem D}Let $K$ be a field of $\Char(K)\neq2$.
\begin{enumerate}
    \item If $d\in\{1,2\}$, then any two embeddings $\psi,\psi'\colon \pp^1 \hookrightarrow \pp^3$ over $K$ of degree $d$ are algebraically isotopic.
    \item The set of algebraic isotopy classes of embeddings $\pp^1 \hookrightarrow \pp^3$ over $K$ of degree $3$ is in bijection to $K^\times / K^{\times4}$.
    \item Two embeddings $\psi,\psi'\colon \pp^1 \hookrightarrow \pp^3$ over $K$ of degree $4$ are algebraically isotopic if and only if they have the same arithmetic writhe.
\end{enumerate}
\end{namedthm*}

In the case $K=\R$, Theorem D strengthens a result by Bj{\"o}rklund \cite{bjoerk} who proved that the encomplexed writhe characterizes real rational curves of degree at most four up to \emph{rigid isotopy}, a (potentially strictly) coarser equivalence relation than algebraic isotopy. While Bj{\"o}rklund's result holds true in degree five as well, we do not know whether this is the case for our Theorem D. It was also shown in \cite{bjoerk} that the encomplexed writhe is not enough to distinguish rigid isotopy classes of embeddings $\pp^1 \hookrightarrow \pp^3$ of degree at least six, hence the corresponding statement also fails for the arithmetic writhe and algebraic isotopies. However, Ulrich sheaves other than those produced in Theorem C can be used to give further isotopy invariants. In \Cref{thm:ulrichsecantrational}, we construct a higher rank Ulrich sheaf on the secant variety of a rational normal curve. From this we obtain a new invariant for algebraic isotopy classes of rational curves of degree six, and we construct an example of two embeddings of the same arithmetic writhe for which the new invariant is different, see \Cref{ex:newinvariants}. We do not know whether this is enough to distinguish all isotopy classes, nor do we have a local description as for the arithmetic writhe.
\vspace{10pt}

The paper is organized as follows. In Section \ref{sec:ulrichprelim}, we collect some facts on Ulrich sheaves, and in Section \ref{sec:orientations} we show how they are connected to the $\A^1$-degree, proving Theorem A in \Cref{thm:ulrichdegreewell} and Theorem B in both \Cref{thm:inducedorient} and \Cref{thm:a1chowmatrix}. In Section \ref{sec:prelimcurves}, we turn our attention on secants of curves and on Ulrich sheaves on them: we prove Theorem C in \Cref{thm:ulrichsecant} and we also construct in \Cref{thm:ulrichsecantrational} higher rank Ulrich sheaves on secants of rational normal curves. In Section \ref{sec:writhe}, we define the arithmetic writhe for an embedded curve, and we also show how to compute it via a sum of explicit local writhes, see Remark \ref{rem:localwrithesum}. Finally, in Section \ref{sec:isotopies} we collect some general results about isotopies for projective embeddings, which imply the first two parts of Theorem D. In Section \ref{sec:isospacecurves} we focus on the case of curves in $3$-space. We define the arithmetic writhe of an embedding in $\pp^3$ and we prove that it is invariant under algebraic isotopies, see \Cref{rem:writheemb}. We then prove Theorem D in Theorem \ref{thm:translation} and Corollary \ref{cor:rat4inj}, and we conclude with some remarks on rational curves of higher degree.

\section{Notation and conventions}
Let $K$ always denote a field. By a $K$-variety we mean an integral, separated scheme of finite type over $K$. We will use the word curve to denote a projective, non-singular $K$-variety of dimension $1$. If $X$ is a scheme and $x\in X$, then we denote by $\kappa(x)$ the residue field of $X$ at $x$. If $A,B$ are two divisors on a smooth variety $X$ we write $A\sim B$ to denote that $A$ is linearly equivalent to $B$.

\section{Ulrich sheaves}\label{sec:ulrichprelim}
We first recall the basic properties of Ulrich sheaves as introduced in \cite{ES03}. Consider the projective space $\mathbb{P}^n$ and the corresponding homogeneous coordinate ring $S=K[x_0,\dots,x_n]$. Let $\mathcal{F}$ be a coherent sheaf on $\mathbb{P}^n$ with scheme-theoretic support $\iota\colon X \hookrightarrow \mathbb{P}^n$ of pure dimension $k>0$ and codimension $c=n-k$. Let also $\Gamma_*(\mathcal{F}) = \oplus H^0(\mathcal{F}(q))$ be the module of twisted global sections, seen as a graded $S$-module. The following important equivalence was proven in \cite[Theorem~2.1]{ES03}:  
\begin{Def}[Ulrich sheaf]\label{prop:ulrichchara}
	The sheaf $\mathcal{F}$ is called an \emph{Ulrich sheaf} if it satisfies one of the following equivalent conditions:
\begin{enumerate}[(i)]
	\item $\Gamma_*(\mathcal{F})$ has a linear minimal free resolution as an $S$-module:
	\[  \mathbf{F}: \qquad  0 \longrightarrow F_c \overset{\varphi_c}{\longrightarrow} F_{c-1} \overset{\varphi_{c-1}}{\longrightarrow} \dots \overset{\varphi_2}{\longrightarrow} F_1 \overset{\varphi_1}{\longrightarrow}  F_0 \longrightarrow \mathcal{F} \longrightarrow 0 \]
	where $F_i$ is a direct sum of copies of $S(-i)$. 
	\item $H^i(X,\cF(-i))=0$ for $i>0$ and $H^i(X,\cF(-i-1)=0$ for $i<k$.
	\item if $\pi\colon X\to\pp^k$ is a finite surjective linear projection then $\pi_*\iota^*\cF\cong  \cO_{\pp^k}^{\oplus t}$ for a certain $t>0$.
\end{enumerate}
\end{Def} 
The resolution of $\mathcal{F}$ can be used to compute a \emph{Chow form} of $X$. Recall that this is a polynomial in the Plücker coordinates of the Grassmannian $\G(c-1,n)$ that cuts out the locus of $(c-1)$-planes that intersect $X$. In the coordinate ring of $\G(c-1,n)$ it is unique up to a scalar factor. A power of the Chow form can be written as the determinant of a matrix with entries linear forms in the Plücker coordinates \cite[\S3]{ES03}. This is constructed from the resolution $\mathbf{F}$ as follows: after choosing a basis of each $F_i$ the maps $\varphi_i$ are given by matrices $A_i$ whose entries are linear forms in the variables $x_0,\ldots,x_n$. We consider these linear forms as degree one elements of the tensor algebra $\textnormal{T}(K^{n+1})^\vee$ and define $\gamma(\textbf{F})$ to be the product $A_1\cdots A_{c}$ over $\textnormal{T}(K^{n+1})^\vee$. The entries of $\gamma(\textbf{F})$ are multilinear forms on $K^{n+1}$ and since $\varphi_i\circ\varphi_{i+1}=0$ these multilinear forms are alternating. Therefore, the entries of $\gamma(\textbf{F})$ are elements of $\wedge^{c}(K^{n+1})^\vee$ and thus linear in the Plücker coordinates of $\G(c-1,n)$. Evaluating $\gamma(\textbf{F})$ at a $(c-1)$-plane $V\subset\pp^n$ gives a singular matrix if and only if $V$ intersects $X$. This implies that the determinant of $\gamma(\textbf{F})$ is the $\rank(\cF)$-th power of the Chow form of $X$. It follows from the definition of the matrix $\gamma(\textbf{F})$ that it actually depends only on the choice of bases of $F_0$ and $F_c$. Thus $\gamma(\textbf{F})$ is uniquely determined by $\cF$ up to multiplication from left and right by invertible matrices over $K$. Finally, if there exists a symmetric isomorphism $\cF\to\shExt^c(\cF,\cO_{\pp^n})(-c)$, then by choosing a suitable basis the determinantal representation can be made symmetric \cite[\S3.1]{ES03}. We will come back to the symmetric case later.

\begin{ex}\label{ex:twisted1}
	We consider the usual twisted cubic curve $\iota\colon \pp^1 \hookrightarrow \pp^3$ given by  $\iota(s:t)=(s^3:s^2t:st^2:t^3)$.
	By \Cref{prop:ulrichchara} and a direct computation of cohomology the sheaf $\cF=\iota_*\left(\cO_{\pp^1}(2)\right)$ is Ulrich. A quick computation using the computer algebra system \texttt{Macaulay2} \cite{M2} gives us the  minimal resolution
	\begin{equation*}
	\begin{tikzcd}
	\textbf{F}\colon&0\arrow{r} & S(-2)^3\arrow{r}{\varphi_2}& S(-1)^6\arrow{r}{\varphi_1}&S^3\arrow{r}&\Gamma_*(\mathcal{F})\arrow{r}&0.
	\end{tikzcd}
	\end{equation*}
	where the maps $\varphi_1$ and $\varphi_2$ are given by the matrices
	\begin{scriptsize} 
	\begin{equation*}
	A_1=\begin{pmatrix}-x_1& -x_2& -x_2& -x_3& -x_3& 0\\ x_0& 0& x_1& 0& x_2& -x_3\\ 0& x_0& 0& x_1& 0& x_2\end{pmatrix}, \quad A_2=\begin{pmatrix}-x_2& x_1& 0& -x_0& x_0& 0\\ -x_3& x_2& -x_2& 0& x_1& -x_0\\ 0& 0& x_3& -x_2& 0& x_1\end{pmatrix}^t. 
	\end{equation*}
	\end{scriptsize}
	and one calculates that
	\begin{equation*}
	\gamma(\textbf{F})=\begin{pmatrix}-x_{12}& -x_{13}& x_{23} \\ x_{02}& x_{12}+x_{03}& -x_{13}\\ -x_{01}& -x_{02}& x_{12}\end{pmatrix} 
	\end{equation*}
	where the $x_{ij}$ are the usual  Plücker coordinates on $\G(1,3)$.
\end{ex}

\section{Ulrich sheaves and the \texorpdfstring{$\mathbb{A}^1$}{A1}-degree}\label{sec:orientations}
In this section we discuss a connection of Ulrich sheaves and a recently developed notion of degree in the context of $\A^1$-enumerative geometry. 
\subsection{Orientations and \texorpdfstring{$\mathbb{A}^1$}{A1}-degree}\label{sec:orient}
In this subsection we recall some preliminaries from $\A^1$-enumerative geometry, mostly following the exposition of \cite[\S8]{pauliwickelgren}. Let $K$ always denote a field.

\begin{Def}[Algebraic orientation]\label{def:orientation}
	Let $X$ be a non-singular $K$-variety. An \emph{algebraic orientation} of $X$ is an isomorphism $L\otimes L\to\omega_X$ where $L$ is a line bundle on $X$ and $\omega_X$ the canonical sheaf. We also say that $X$ is \emph{oriented} by $L\otimes L\to\omega_X$. A \emph{theta characteristic} of $X$ is a line bundle $L$ such that $L\otimes L$ is isomorphic to $\omega_X$.
\end{Def}

\begin{ex}
	If $X=\pp^{n}$ with $n=2h-1$ odd, then $\cO_{\pp^n}(-h)$ is a theta characteristic of $X$. If $n$ is even, then $\pp^n$ does not have a theta characteristic.
\end{ex}

\begin{Def}\label{def:equivorient}
	Let $X$ be a non-singular $K$-variety and $L$ a theta characteristic. Two algebraic orientations $\psi_1,\psi_2\colon L\otimes L\to\omega_X$ are \emph{equivalent} if $\psi_1\circ \psi_2^{-1}$ is multiplication by a square of a global section of $\cO_X$.
\end{Def}

\begin{rem}\label{rem:orientR}
	Let $K=\R$ and $X$ be a non-singular $\R$-variety of dimension $n$. Let $\psi\colon L\otimes L\to\omega_X$ be an algebraic orientation. Let $f$ some rational section of $L$ and $\omega=\psi(f\otimes f)$. All zeros and poles of the rational differential $n$-form $\omega$ on $X$ are of even order. It can be shown that this implies the existence of a rational function $g$ on $X$ which is nonnegative wherever it is defined on $X(\R)$ and such that $g\cdot\omega$ does not have any zeros or poles on $X(\R)$. Thus $g\cdot\omega$ defines an orientation on $X(\R)$ in the classical topological sense. It is straight-forward to check that this does not depend on the choices we made and that two algebraic orientations that are equivalent in the sense of \Cref{def:equivorient} induce the same classical orientation on $X(\R)$.
\end{rem}

\begin{ex}\label{ex:thetachars}
	Let $X$ be a curve of genus $g$ over $\R$. Then every topological orientation of $X(\R)$ is induced by an algebraic orientation $L\otimes L\to\omega_X$. This follows for instance from \cite[Satz~2.4.c]{geyer} or by the same argument as in \cite[Corollary~2.4]{totthet}. However, this is not true in general. For instance, let $X$ be a real K3 surface: since the Picard group of $X$ is torsion-free, there is, up to a scalar multiple, only one algebraic orientation on $X$. Thus only two topological orientations on $X(\R)$ arise in this way. However, if $X(\R)$ is not connected, then there exist at least four different topological orientations of $X(\R)$.
\end{ex}

There is a relative version of algebraic orientations for morphisms \cite[Definition~7]{pauliwickelgren}:

\begin{Def}[Relative orientation]
	Let $f\colon X\to Y$ be a finite surjective morphism of non-singular $K$-varieties. A \emph{relative orientation} of $f$ is an isomorphism 
	$$\psi\colon L\otimes L\to\shHom(\det\cT_X, f^*\det\cT_Y)$$ 
	where $L$ is a line bundle on $X$ and $\cT_X,\cT_Y$ are the tangent bundles on $X$ and $Y$ respectively. Two relative orientations $\psi_1,\psi_2$ as above 
	are \emph{equivalent} if $\psi_1\circ \psi_2^{-1}$ is multiplication by a square of a global section of $\cO_X$.
\end{Def}

\begin{rem}
	If $f\colon X\to Y$ is a finite surjective morphism of non-singular $K$-varieties with algebraic orientations $\psi_1\colon L_1\otimes L_1\to\omega_X$ and $\psi_2\colon L_2\otimes L_2\to\omega_Y$, then a relative orientation of $f$ is given by the induced isomorphism 
	\begin{equation*}
	\psi\colon (L_1\otimes f^*L_2^{\vee})^{\otimes 2}\to \omega_X \otimes f^*\omega_Y^{\vee}.    
	\end{equation*}
	Algebraic orientations $\psi_1'$ and $\psi_2'$ on $X$ and $Y$ that are equivalent to $\psi_1$ and $\psi_2$ induce a relative orientation $\psi'$ for $f$ that is equivalent to $\psi$.
\end{rem}

Recall the set of equivalence classes of non-degenerate symmetric bilinear forms on finite dimensional vector spaces over $K$ form a monoid via the  orthogonal sum. The Grothendieck--Witt group $\GW(K)$ of $K$ is the Grothendieck group of this monoid. It is generated by all equivalence classes $\langle a\rangle$ of one-dimensional bilinear forms \begin{equation*}
K\times K\to K,\quad (x,y)\mapsto axy
\end{equation*}
for $a\in K^\times$. Note that $\langle a\rangle=\langle ab^2\rangle$ for all $b\in K^\times$.  

For a finite morphism $f\colon X\to Y$ of non-singular oriented $K$-varieties of the same dimension one can define a notion of degree that takes its values in $\GW(K)$ under appropriate
hypotheses. Actually it suffices to make the weaker assumption of requiring $f$ to have a relative orientation $\psi\colon L\otimes L\to\shHom(\det\cT_X, f^*\det\cT_Y)$. In this case, let $x\in X$ be a closed point where $f$ is not ramified and set $y=f(x)\in Y$. Note that this in particular requires the field extension $\kappa(x)/\kappa(y)$ to be separable. The differential of $f$ defines a morphism $\T f\colon \cT_X \to f^*\cT_Y$ and thus $\kappa(x)$-linear maps:
\[ \T_xf\colon\T_xX\to \T_yY\otimes_{\kappa(y)}\kappa(x), \, \J_xf = \det \T_xf\colon \det \T_x X \to \det \T_yY\otimes_{\kappa(y)} \kappa(x) \]
Two bases of $\T_xX$ and $\T_yY$ are called compatible with respect to the relative orientation if the determinant of the linear map that sends one base to the other is an element in the fiber of $\shHom(\det\cT_X, f^*\det\cT_Y)$ at $x$ which is the image of a square under $\psi$ \cite[Definition~8]{pauliwickelgren}.  After choosing such compatible bases, we can identify the determinant $\J_xf$ with an element in $\kappa(x)$. Thanks to the requirement of compatibility, it is straight-forward to check that the class $\langle \J_xf\rangle\in\GW(\kappa(x))$ does not depend on the bases. We can also define $\langle \J_xf \rangle$ without mentioning bases as follows: taking the fiber of the orientation at $x$ we see that $\J_x f = g\cdot \psi(t\otimes t)$, for a certain $t\in L_x$ and $g\in \kappa(x)$. Then $\langle \J_xf \rangle = \langle g \rangle$.

Then the \emph{local $\A^1$-degree $\deg_x^{\A^1}(f)$ of $f$ at $x$} as the class in $\GW(\kappa(y))$ of the symmetric $K$-bilinear form defined by \begin{equation}\label{eq:loca1}
\Tr_{\kappa(x)/\kappa(y)}\langle \J_xf\rangle\colon\,\kappa(x)\times\kappa(x)\to \kappa(y),\, (a,b)\mapsto\Tr_{\kappa(x)/\kappa(y)}(\J_xab).
\end{equation}
Finally, for a closed point $y\in Y$ outside the branch locus of $f$, one defines the \emph{$\A^1$-degree $\deg_y^{\A^1}(f)$ of $f$ at $y$} as
\begin{equation*}
\deg_y^{\A^1}(f)=\sum_{x\in f^{-1}(y)}\deg_x^{\A^1}(f) \in \GW(\kappa(y)).
\end{equation*}

\begin{rem}
 One could argue that \Cref{eq:loca1} is rather a \emph{formula} for the local $\A^1$ degree \cite{morelmadrid,kass} than a definition.
\end{rem}

\begin{rem}
	The (local) $\A^1$-degree only depends on the equivalence class of the relative orientation.
\end{rem}

\begin{rem}
	In the case $K=\R$ the signature of the $\A^1$-degree of $f\colon X\to Y$ is the topological degree of the restriction of $f$ to the real parts of $X$ and $Y$. This is clear from the description of the local degree in \Cref{eq:loca1} together with the observation that every scaled trace bilinear form $\C\times\C\to\R$ has signature zero.
\end{rem}

\begin{Def}\label{def:welldef}
	Let $f\colon X\to Y$ be a relatively oriented finite surjective morphism of non-singular $K$-varieties. If there is an element $d\in\GW(K)$ such that for every closed point $y\in Y$ outside the branch locus we have that \begin{equation*}
	\deg_{y}^{\A^1}(f)=d\otimes_K \kappa(y)\in \GW(\kappa(y)),
	\end{equation*}
	then we say that $\deg^{\A^1}(f)$ is \emph{well-defined} and we write $\deg^{\A^1}(f)=d$.
\end{Def}

\begin{rem}
	Note that in the situation of \Cref{def:welldef} writing $\deg^{\A^1}(f)=d$ is a slight abuse of notation since there might be several different $d\in\GW(K)$ satisfying the requirements from \Cref{def:welldef} --- for instance in the case that $Y$ has no $K$-rational points.
\end{rem}

The degree of a morphism is not always well-defined. An important case where this happens is when the base of the morphism is $\A^1$ chain-connected:

\begin{ex} Let $f\colon X\to Y$ as in \Cref{def:welldef}.
	\begin{enumerate}
		\item If $Y$ is \emph{$\A^1$-chain connected} in the sense of \cite[Definition~9]{pauliwickelgren}, then $\deg^{\A^1}(f)$ is well-defined by \cite[Theorem~8]{pauliwickelgren}. This remains true under the potentially weaker hypothesis that $Y$ is \emph{$\A^1$-connected} \cite[\S2.5]{KLSW}.
		\item If $K=\R$, then it suffices that $Y(\R)$ is connected in the classical topology because the (signature of the) $\A^1$-degree is the  Brouwer degree of the map $X(\R)\to Y(\R)$.
		\item In general $\deg^{\A^1}(f)$ is not necessarily well-defined, see for example \cite[Example~16]{pauliwickelgren}.
	\end{enumerate}
\end{ex}

\subsection{Relative Ulrich sheaves and the \texorpdfstring{$\mathbb{A}^1$}{A1}-degree}
We now derive a further case when $\deg^{\A^1}(f)$ is well-defined which is a condition on the relative orientation rather than on the target. Motivated by \Cref{prop:ulrichchara}(iii) one makes the following definition.

\begin{Def}
	Let $f\colon X\to Y$ a morphism of schemes. A coherent sheaf $\cF$ on $X$ is called \emph{$f$-Ulrich} if there exists an integer $t>0$ such that $f_*\cF\cong\cO_Y^{\oplus t}$.
\end{Def}

We further recall the following construction of a trace map for differential forms.

\begin{cons}\cite[Exercise~III.7.2]{Hart77}\label{cons:trace}
	Let $f\colon X\to Y$ be a finite surjective morphism of non-singular varieties over $K$. We recall the construction of a natural trace map $t\colon f_*\omega_X\to\omega_Y$. First let $X=\Spec(B)$ and $Y=\Spec(A)$ be affine, $\dim(X)=\dim(Y)=n$ and $E=\Quot(A)$, $F=\Quot(B)$ the fields of fractions. Choose any nonzero $\omega\in\wedge^n\Omega_{A/K}$. Then, for any $\omega'\in\wedge^n\Omega_{B/K}$, there is a $b\in F$ such that $\omega'=b\cdot\omega$. The trace map sends $\omega'$ to $t(\omega')=\tr_{F/E}(b)\cdot\omega$. Note that while $b$ is not necessarily in $B$, we always have that $t(\omega')\in \wedge^n\Omega_{A/K}$. In the general, not necessarily affine case the map $t$ is obtained by glueing. 
\end{cons}

Now let $f\colon X\to Y$ be a finite surjective morphism of non-singular $K$-varieties, relatively oriented by the isomorphism 
\begin{equation*}
\psi\colon L\otimes L\to\shHom(\det\cT_X, f^*\det\cT_Y)=\omega_X\otimes f^*\omega_Y^{-1}.
\end{equation*}
By the projection formula the push-forward of $\psi$ induces a map
\begin{equation*}
f_*L\otimes f_*L\to f_*\omega_X\otimes \omega_Y^{-1}.   
\end{equation*}
Tensoring the trace morphism from \Cref{cons:trace} by $\omega_Y^{-1}$ we obtain a morphism $f_*\omega_X\otimes \omega_Y^{-1}\to\cO_Y$ which we precompose with the map above to obtain a symmetric $\mathcal{O}_Y$-bilinear form $\tilde{\psi}\colon f_*L\otimes f_*L\to\cO_Y$. The following simple lemma can also be seen as a consequence of \cite[Corollary 3.10]{KLSW}, but we give here a self-contained proof. 

\begin{lem}\label{lem:degreewell}
	Let $f\colon X\to Y$ be a finite surjective morphism of non-singular $K$-varieties, relatively oriented by the isomorphism 
	\begin{equation*}
	\psi\colon L\otimes L\to\shHom(\det\cT_X, f^*\det\cT_Y).
	\end{equation*}
	Let $y\in Y$ a closed point outside the branch locus.
	The class of the fiber of $\tilde{\psi}$ at $y$ in $\GW(\kappa(y))$ is $\deg_y^{\A^1}(f)$.
\end{lem}

\begin{proof}
	Let $U\subset Y$ an open affine neighborhood of $y$ and $V=f^{-1}(U)$. Then we have a finite ring extension $A\subset B$ where $A=\cO_Y(U)$ and $B=\cO_X(V)$. If we choose $U$ sufficiently small, then $\Omega_{A/K}$ is a free $A$-module and $\Omega_{B/K}$ and $M=L(V)$ are free $B$-modules. Since $y$ is not in the branch locus of $f$ we can, after further shrinking $U$ if necessary, moreover assume that $\Omega_{B/K}=\Omega_{A/K}\otimes_A B$. Then there are $a_1,\ldots,a_n\in A$ such that $da_1,\ldots,da_n$ is an $A$-basis of $\Omega_{A/K}$ and $da_1\otimes 1,\ldots,da_n\otimes 1$ is a $B$-basis of $\Omega_{B/K}$. Let $v_1,\ldots,v_n\in\Omega_{A/K}^\vee$ and $w_1,\ldots,w_n\in \Omega_{B/K}^\vee$ the dual bases so that $\T_x f(w_i)=v_i$ for all $x\in V$, and let $t\in M$ a generator of $M$. Then 
	\begin{equation*}
	\psi(t\otimes t)=g\cdot \left((da_1\otimes 1)\wedge \cdots\wedge (da_n\otimes 1)\right)\otimes (v_1\wedge\cdots\wedge v_n)
	\end{equation*}
	for some $g\in B^\times$. Then it follows that $\langle \J_{x_i}f \rangle = \langle g(x_i) \rangle$. 
 For $s_1,s_2\in M$ we can write $s_i=h_i\cdot t$ for $i=1,2$ and $h_i\in B$. By the definition of $\tilde{\psi}$ we then have
	\begin{equation}\label{eq:psibar}
	\tilde{\psi}(s_1,s_2)=\Tr_{B/A}(g\cdot h_1\cdot h_2).
	\end{equation}
	Since $B\otimes_A\kappa(y)=\kappa(x_1)\oplus\cdots\oplus\kappa(x_r)$, \Cref{eq:psibar} evaluated at $y$ equals
	\begin{equation*}
	\sum_{i=1}^r \Tr_{\kappa(x_i)/\kappa(y)}(g(x_i)\cdot h_1(x_i)\cdot h_2(x_i)).
	\end{equation*}
	Therefore, the class of the fiber of $\tilde{\psi}$ at $y$ equals
	\begin{equation*}
	\sum_{i=1}^r \Tr_{\kappa(x_i)/\kappa(y)}\langle g(x_i)\rangle=\sum_{i=1}^r \Tr_{\kappa(x_i)/\kappa(y)}\langle \J_{x_i}f\rangle=\sum_{i=1}^r\deg_{x_i}^{\A^1}(f)=\deg_y^{\A^1}(f)
	\end{equation*}
	where we use that $\langle g(x_i)\rangle=\langle \J_{x_i}f\rangle$.
\end{proof}

Passing from $\tilde{\psi}\colon f_*L\otimes f_*L\to\cO_Y$ to global sections, we obtain a symmetric bilinear form 
\begin{equation*}
\bar{\psi}\colon  H^0(Y,f_*L)\times H^0(Y,f_*L)\to H^0(Y,\cO_Y).
\end{equation*}
In the case that $H^0(Y,\cO_Y)=K$, this is a symmetric $K$-valued bilinear form.

\begin{thm}\label{thm:ulrichdegreewell}
	Let $f\colon X\to Y$ be a finite surjective morphism of non-singular $K$-varieties, relatively oriented by the isomorphism 
	\begin{equation*}
	\psi\colon L\otimes L\to\shHom(\det\cT_X, f^*\det\cT_Y).
	\end{equation*}
	If $L$ is $f$-Ulrich and $H^0(Y,\cO_Y)=K$, then the $K$-bilinear form
	\begin{equation*}
	\bar{\psi}\colon H^0(Y,f_*L)\times H^0(Y,f_*L)\to K.
	\end{equation*}
	is non-degenerate. In this case $\deg^{\A^1}(f)$ is well-defined and equals to the class of $\bar{\psi}$. 
\end{thm}

\begin{proof}
	At every $y\in Y$ the fiber of a $K$-basis of $H^0(Y,f_*L)$ is a $\kappa(y)$ basis of the fiber of $f_*L$ at $y$ because $L$ is $f$-Ulrich. Therefore, the fiber of $\tilde{\psi}$ at $y$ is isometric to $\bar{\psi}\otimes_K\kappa(y)$. Now the claim follows from \Cref{lem:degreewell}.
\end{proof}

\begin{ex}
 Let $X$ be a smooth projective curve and $f\colon X\to\pp^1$ a finite surjective morphism. Every relative orientation of $f$ is given by a line bundle of the form $L=M(1)$ where $M$ is a theta characteristic on $X$. Then $L$ is $f$-Ulrich if and only if $h^0(M)=0$. In particular, for $X=\pp^1$ every relative orientation is given by an Ulrich line bundle.
\end{ex}

One has the following converse of \Cref{lem:degreewell}.

\begin{lem}
	Let $f\colon X\to Y$ be a finite surjective morphism of non-singular complete $K$-varieties, relatively oriented by the isomorphism 
	\begin{equation*}
	\psi\colon L\otimes L\to\shHom(\det\cT_X, f^*\det\cT_Y)
	\end{equation*}
	and assume that $Y$ is geometrically irreducible. The following are equivalent:
	\begin{enumerate}
		\item $L$ is $f$-Ulrich.
		\item $\bar{\psi}$ is non-degenerate and for every closed point $y\in Y$ one has 
		\begin{equation*}
		\deg_{y}^{\A^1}(f)=\bar{\psi}\otimes_K \kappa(y)\in \GW(\kappa(y)).   
		\end{equation*}
	\end{enumerate}
\end{lem}

\begin{proof}
	
 \Cref{thm:ulrichdegreewell} proves that $(1)$ implies $(2)$.  On the other hand, condition $(2)$ implies that $\bar{\psi}$ is non-degenerate and for $y\in Y$
	\begin{equation*}
	\dim(H^0(X,L))=\rank(\bar{\psi})=\rank(\deg_{y}^{\A^1}(f))=\deg(f).
	\end{equation*}
	Therefore, by \cite[Theorem~4.8]{hanselka2020positive} the line bundle $L$ is $f$-Ulrich.
\end{proof}

We conclude this section with two examples showing that the two sufficient conditions for the $\A^1$-degree being well-defined, that we have seen here, do not imply each other. 

\begin{ex}
	Consider a smooth plane cubic curve $X\subset\pp^2$ and let $f\colon X\to\pp^1$ the linear projection from a point not on $X$. Then $f$ is relatively oriented by 
	\begin{equation*}
	\psi\colon \cO_X(1)\otimes\cO_X(1)\to \cO_X(2)\cong\shHom(\det\cT_X, f^*\det\cT_{\pp^1}).
	\end{equation*}
	However, the line bundle $\cO_X(1)$ is not $f$-Ulrich because it violates part (ii) of \Cref{prop:ulrichchara}. Since $\dim(H^0(X,\cO_X(1))=3=\deg(f)$, this implies in particular that $\bar{\psi}$ cannot be a non-degenerate bilinear form. On the other hand, the target $\pp^1$ is clearly $\A^1$-chain connected.
\end{ex}

\begin{ex}
	Assume here that the characteristic of $K$ is not $2$ or $3$, and consider the smooth sextic curve $X$ in $\pp^3$ defined over $\Q$ by the two equations
	\begin{equation*}
	Q=x_1^2+x_2^2-x_3^2\textnormal{ and }T=x_0^3+x_1^3+x_2^3
	\end{equation*}
	and the smooth cubic curve $Y$ in $\pp^2$ defined by $T$. The linear projection from the point $[0,0,0,1]$ defines a finite surjective morphism $f\colon X\to Y$ of degree two.
	The ramification divisor of $f$ is the zero divisor of $x_3$ on $X$. This shows that 
	\begin{equation*}
	\shHom(\det\cT_X, f^*\det\cT_Y)\cong\cO_X(1).
	\end{equation*}
	Furthermore, the zero divisor of $x_2-x_3$ on $X$ is of the form $2D$ where $D$ is an effective divisor of degree three because this hyperplane intersects the singular quadric in $\pp^3$ defined by $Q$ in a line with multiplicity two. Thus the corresponding line bundle $L$ satisfies $L\otimes L\cong\cO_X(1)$. In particular, we obtain a relative orientation
	\begin{equation*}
	L\otimes L\to \shHom(\det\cT_X, f^*\det\cT_Y)
	\end{equation*}
	of $f$. A calculation with the computer algebra system \texttt{Macaulay2} \cite{M2} further shows that $f_*L=\cO_Y^{\oplus 2}$. Thus $\deg^{\A^1}(f)$ is well-defined by \Cref{thm:ulrichdegreewell}. 	On the other hand, by L\"uroth's theorem there is no non-constant morphism $\A^1\to Y$ and therefore $Y$ is not $\A^1$-chain connected (actually not even $\A^1$-connected \cite[Corollary~2.4.4]{asokmorel}).
    Let us compute the degree of this map using the method of Theorem \ref{thm:ulrichdegreewell}. The quadratic form $\bar{\psi}$ is given by the composition:
    \begin{eqnarray*}
    H^0(X,\mathcal{O}_X(D))\otimes H^0(X,\mathcal{O}_X(D)) &\longrightarrow& H^0(X,\mathcal{O}_{X}(2D))\\ &\overset{\sim}{\longrightarrow}& H^0(X,\mathcal{O}_X(R))\\ &\longrightarrow& H^0(Y,\mathcal{O}_Y)
    \end{eqnarray*}
    where the first map is the multiplication map, the second map is given by the linear equivalence between $2D$ and the ramification divisor $R=\{x_3=0\}$ and the last map is the trace. A basis of $H^0(X,\mathcal{O}_X(D))$ is given by the rational functions $1,\frac{x_1}{x_2-x_3}$ and multiplying them together we obtain the rational functions $1,\frac{x_1}{x_2-x_3}, \left(\frac{x_1}{x_2-x_3} \right)^2$. These get mapped to the elements $\frac{x_2-x_3}{x_3},\frac{x_1}{x_3},\frac{x_1^2}{(x_2-x_3)x_3}$ in $H^0(X,\mathcal{O}_X(R))$, and finally the traces of these are
    \begin{multline*}
        \Tr\left( \frac{x_2-x_3}{x_3} \right) = \frac{x_2-x_3}{x_3} - \frac{x_2+x_3}{x_3} = \frac{-2x_3}{x_3} = -2, \quad
        \Tr\left( \frac{x_1}{x_3} \right) = \frac{x_1}{x_3} - \frac{x_1}{x_3} = 0 \\
        \Tr\left( \frac{x_1^2}{(x_2-x_3)x_3} \right) = \frac{x_1^2}{(x_2-x_3)x_3} - \frac{x_1^2}{(x_2+x_3)x_3} = \frac{2x_1^2}{x_2^2-x_3^2} = -2 
    \end{multline*}
    where the last equality comes from the identity $x_1^2 = x_3^2-x_2^2$ on $X$. In conclusion, the quadratic form is given by $(\lambda,\mu) \mapsto -2\lambda^2 - 2 \mu^2$ so that the $\mathbb{A}^1$-degree is $\langle -2 \rangle + \langle -2 \rangle$.
\end{ex}

\subsection{Symmetric Ulrich sheaves}
Let us first set up some notation. We let $S=\Gamma_*(\cO_{\pp^n})=\oplus_{j\in\Z}\Gamma(\pp^n,\cO_{\pp^n}(j))$. If $M$ is a graded $S$-module, we denote by $\tilde{M}$ the sheaf associated to $M$ on $\pp^n$. If $\varphi\colon M\to N$ is a homomorphism of graded $S$-modules, then we denote by $\tilde{\varphi}\colon \tilde{M}\to\tilde{N}$ the associated morphism of quasi-coherent sheaves on $\pp^n$. One has 
\begin{equation*}
\cF=\widetilde{\Gamma_*(\cF)}   
\end{equation*}
for any quasi-coherent sheaf $\cF$ on $\pp^n$ \cite[Proposition~II.5.15]{Hart77}. Conversely, if $M$ is a finitely generated graded $S$-module of $\depth(M)\geq2$, one also has $M=\Gamma_*(\tilde{M})$ \cite[Corollary~A.1.13]{Eis05}. This applies in particular when $M$ is Cohen--Macaulay of dimension at least
two.

Now let $\cF$ be an Ulrich sheaf on $\pp^n$ whose support is a closed subvariety $\iota\colon X\hookrightarrow\pp^n$ of dimension at least one. Let us denote the rank of $\iota^*\cF$ by $r$ and by $M=\Gamma_*(\cF)$ the module of twisted global sections over $S$. As explained in \Cref{sec:ulrichprelim} we can construct from $\cF$ a matrix $\gamma(\textbf{F})$ obtained from the free resolution $\textbf{F}$ of $M$ whose entries are linear forms in Plücker coordinates and whose determinant is the $r$-th power of the Chow form of $X$. We now recall from \cite[\S3.1]{ES03} a condition for this matrix to be symmetric. Consider the contravariant functor 
\begin{equation}\label{eq:functorD}
D\colon \cG\mapsto\shExt^c(\cG,\cO_{\pp^n}(-c))
\end{equation}
where $c$ is the codimension of $X$. If $\cF$ is Ulrich supported on $X$, the sheaf $D(\cF)$ is again an  Ulrich sheaf with support $X$ and there is a canonical isomorphism $\beta\colon \cF\to DD\cF$.
A morphism $\sigma\colon \cF\to D\cF$ is \emph{symmetric} if $\sigma=D(\sigma)\circ\beta$. If such a morphism exists, then a suitable choice of bases makes the matrix $\gamma(\textbf{F})$ symmetric \cite[\S3.1]{ES03}.

We will now recall a more explicit description of a minimal free resolution of $M$. To this end recall the following definition. 

\begin{Def}
	Let $A=(A_1, \ldots, A_s)$ be a tuple of pairwise commuting $m \times m$ matrices over a ring $R$. We can define on $R^m$ the structure of an
	$R[t_1,\ldots,t_s]$-module $P$ by letting $t_i$ act on $R^m$ via multiplication with the matrix $A_i$ from the left.
	Consider the complex $P \otimes \kosz(t)$ where $\kosz(t)$ is the Koszul complex of the sequence $t=(t_1,\ldots,t_s)$. We can view this  complex as a complex of $R$-modules instead of $R[t_1,\ldots,t_s]$-modules. 
	The resulting complex of free $R$-modules is called \textit{the Koszul complex associated to the matrices} $A_1, \ldots, A_s$
	and we denote it by $\kosz(A)$.
	The maps of $\kosz(A)$ are obtained from the maps of the Koszul complex $\kosz(t)$ by replacing every occurrence of $t_i$ by $A_i$ for $i=1,\ldots,s$.
\end{Def}

Let $k=n-c$ be the dimension of $X$ and $s_0,\ldots,s_k\in H^0(\pp^n,\cO_{\pp^n}(1))$ be sections that do not have a common zero on $X$. Then the morphism \begin{equation*}
\pi\colon X\to\pp^k,\, x\mapsto [s_0(x),\cdots,s_k(x)]
\end{equation*}
defines a finite surjective morphism. Let $s_0,\ldots,s_n$ a basis of $H^0(\pp^n,\cO_{\pp^n}(1))$ and consider the inclusion
\begin{equation*}
R:=K[s_0,\ldots,s_k]\subset K[s_0,\ldots,s_n]=S
\end{equation*}
of graded polynomial rings. Because $\cF$ is an Ulrich sheaf, it holds that $\pi_*\mathcal{F}\cong \mathcal{O}_{\mathbb{P}^k}^{\oplus t}$, so that there is an isomorphism
\begin{equation*}
\psi\colon M\to R^m
\end{equation*}
of $R$-modules where $m=t\cdot\deg(X)$. Under this identification multiplication by $s_i$ can be represented by an $m\times m$ matrix $B_i$ with entries from $R_1$ for $i=0,\ldots,n$. Let $A_i=s_i\cdot I_m-B_i$ where $I_m$ is the $m\times m$ identity matrix.

\begin{thm}\label{thm:minres}
	The Koszul complex $\kosz(A)$ associated to the matrices $A_{k+1},\ldots,A_n$ is a minimal free resolution of $M$. In particular, we can describe $M$ as the cokernel of the matrix $(A_{k+1}\cdots A_n)$ and $\Ext_S^c(M,S)(-c)$ as the cokernel of $(A^t_{k+1}\cdots A^t_n)$.
\end{thm}

\begin{proof}
	See \cite[Proof of Theorem~6.2.5]{KummerDiss} or \cite[Remark~4.8]{realfib}.
\end{proof}

\begin{rem}
	Note that, since the Koszul complex is self-dual, \Cref{thm:minres} also implies that  a minimal free resolution of $\Ext_S^c(M,S)(-c)$ is given by the Koszul complex $\kosz(A^t)$ associated to the matrices $A^t_{k+1},\ldots,A^t_n$. This is one way to see that $\Ext_S^c(M,S)(-c)$ is an Ulrich module.
\end{rem}

\begin{cor}\label{cor:extshriek}
	There is a natural isomorphism 
	\begin{equation*}
	\Ext_S^c(M,S)(-c)\cong\Hom_R(M,R)
	\end{equation*}
	of $S$-modules. Here the $S$-module structure on $\Hom_R(M,R)$ is defined by $(s\varphi)(x):=\varphi(sx)$ for $s\in S$ and $x\in M$.
\end{cor}

\begin{proof}
	Under the identification 
	\begin{equation*}
	\psi^\vee\colon R^m=\Hom_R(R^m,R)\to \Hom_R(M,R)   
	\end{equation*}
	multiplication by $s_i$ on $\Hom_R(M,S)$ is represented by the $m\times m$ matrix $B^t_i$. Hence the $S$-module $\Hom_R(M,R)$ is isomorphic to the cokernel of $(A^t_{k+1}\cdots A^t_n)$ which is $\Ext_S^c(M,S)(-c)$ by \Cref{thm:minres}. It is straight-forward to check that the resulting isomorphism $\Ext_S^c(M,S)(-c)\to\Hom_R(M,R)$ does not depend on the choice of the $R$-module isomorphism $\psi\colon M\to R^m$.
\end{proof}

The next lemma shows that we can work with $\Ext_S^c(M,S)$ and $\shExt^c(\cF,\cO_{\pp^n})$ interchangeably.

\begin{lem}\label{lem:sheafmodule}
	Let $N=\Ext_S^c(M,S)$ and $\cG=D(\cF)$.
	\begin{enumerate}
		\item There is a natural isomorphism of coherent sheaves
		$\widetilde{N}\cong\cG$.
		\item There is a natural isomorphism of graded $S$-modules $N\cong\Gamma_*(\cG)$.
	\end{enumerate} 
\end{lem}

\begin{proof}
	A minimal free resolution of $M$ gives rise to a free resolution of $\cF$ of length $c$. Thus by \cite[Proposition~III.6.5]{Hart77} we can compute $\shExt^c(\cF,\cO_{\pp^n}(-c))$ as the cokernel of the dual of the last map of this resolution. This map is induced by the dual of the last map of the resolution of $M$ whose cokernel is $\Ext_S^c(M,S)$. This implies part (1). Since $N$ is a twist of an Ulrich module and therefore Cohen--Macaulay, we have $N=\Gamma_*(\widetilde{N})$ which implies part (2).
\end{proof}
Recall that for a finite surjective morphism $f\colon X\to Y$ of Noetherian schemes the right adjoint functor of $f_*$ (considered as a functor from quasi-coherent sheaves on $X$ to quasi-coherent sheaves on $Y$) can be described as follows. For a quasi-coherent $\cO_Y$-module $\cG$ the sheaf $\shHom_Y(f_*\cO_X,\cG)$ is a quasi-coherent $f_*\cO_X$-module. The corresponding quasi-coherent $\cO_X$-module is then $f^!\cG$. See for example \cite[Exercise~III.6.10]{Hart77}. 

\begin{lem}\label{lem:shriekram}
	Let $f\colon X\to Y$ be a finite surjective morphism of smooth varieties over $K$. Then $f^!\cO_Y$ is a line bundle, which is naturally  isomorphic to $\shHom(\det\cT_X, f^*\det\cT_Y)$.
\end{lem}

\begin{proof}
	This is for example shown in \cite[Remark~2.2.19]{KummerDiss}.
\end{proof}

\begin{thm}\label{thm:inducedorient}
	Let $\cF$ be an Ulrich sheaf on $\pp^n$ and $\sigma\colon \cF\to D\cF$ an isomorphism.
	Assume that the support of $\cF$ is a closed subvariety $\iota\colon X\hookrightarrow\pp^n$ of dimension $k=\dim(X)>0$ and that $\iota^*\cF$ has rank one. Let $\pi\colon X\to\pp^k$ be a finite surjective linear projection and $U\subset\pp^k$ an open subset such that $V=\pi^{-1}(U)\subset X$ is smooth. Then there is a natural relative orientation
	\begin{equation*}
	\sigma(\pi|_V)\colon  L\otimes L\to\shHom(\det\cT_V, (\pi|_V)^*\det\cT_{U})
	\end{equation*}
	of $\pi|_V$ where $L=(\iota^*\cF)|_V$.
\end{thm}

\begin{proof}
	\Cref{cor:extshriek} and \Cref{lem:sheafmodule} imply that there is a natural isomorphism
	\begin{equation*}
	\iota^*(D\cF)\to \shHom_X(\iota^*\cF,\pi^!\cO_{\pp^k}).
	\end{equation*}
	Precomposing this with the pullback of $\sigma$ under $\iota$ gives an isomorphism
	\begin{equation*}
	\iota^*(\cF)\to \shHom_X(\iota^*\cF,\pi^!\cO_{\pp^k}).
	\end{equation*}
	Because every Ulrich sheaf is Cohen--Macaulay, the restriction $(\iota^*\cF)|_V$ is a line bundle on $V$. Indeed, in this situation being Cohen--Macaulay implies being locally free by \cite[Proposition~6.1.5]{EGAIV2}.
    Thus restricting to $V$ and tensoring with $\iota^*\cF$ gives the isomorphism
	\begin{equation}\label{eq:extshriek}
	(\iota^*\cF)|_V\otimes(\iota^*\cF)|_V\to (\pi^!\cO_{\pp^k})|_V =  (\pi|_V)^!\cO_{U}.
	\end{equation}
	Finally, by \Cref{lem:shriekram} we obtain on $V$ the isomorphism
	\begin{equation*}
	(\iota^*\cF)|_V\otimes(\iota^*\cF)|_V\to\shHom(\det\cT_V, (\pi|_V)^*\det\cT_U).\qedhere
	\end{equation*}
\end{proof}

\begin{rem}
One can prove in a similar way that if $\cF$ is a rank one sheaf on $X$ with an isomorphism $\sigma\colon\cF\overset{\sim}{\to} D\cF$, and such that the restriction $\cF_{|V}$ is a line bundle $L$, then $L$ is a relative orientation for the map $\pi_{|V}\colon V\to U$. This works also when $\cF$ is not Ulrich. However, when $\cF$ is Ulrich, the proof of \Cref{thm:inducedorient} shows how to compute the $\mathbb{A}^1$-degree of the map $\pi_{|V}$ via Theorem \ref{thm:ulrichdegreewell}. We will make this clear in Theorem \ref{thm:a1chowmatrix}.
\end{rem}

\begin{Def}
	In the situation of \Cref{thm:inducedorient} we call $\sigma(\pi|_V)$ the \emph{relative orientation induced by $\sigma$}.
\end{Def}

\begin{lem}\label{lem:symm}
	If the rank of $\iota^*\cF$ is one, then every isomorphism $\sigma\colon \cF\to D\cF$ is symmetric.
\end{lem}

\begin{proof}
	Since $\cF$ is supported on $X$, it suffices to prove the pullback of the equality $\sigma=D(\sigma)\circ\beta$ under $\iota$. Let $V$ as in \Cref{thm:inducedorient}. Since $V$ is dense in $X$, it further suffices to prove the equality on $V$. As in \Cref{eq:extshriek} in the proof of \Cref{thm:inducedorient} the isomorphism $\sigma$ induces an isomorphism 
	\begin{equation*}
	\rho\colon (\iota^*\cF)|_V\otimes(\iota^*\cF)|_V\to (\pi|_V)^!\cO_{U}
	\end{equation*}
	and the condition on $\sigma$ being symmetric translates to the condition that $\rho_x(a\otimes b)=\rho_x(b\otimes a)$ for all $x\in V$ and $a,b\in(\iota^*\cF)|_x$. Since $(\iota^*\cF)|_V$ is locally free of rank one, there is $t\in (\iota^*\cF)|_x$ and $f,g\in\cO_{X,x}$ such that $a=f\cdot t$ and $b=g\cdot t$. Then
	\begin{equation*}
	\rho_x(a\otimes b)=fg\cdot\rho_x(t\otimes t)=\rho_x(b\otimes a).\qedhere
	\end{equation*}
\end{proof}

Now we are ready to state the main result of this section.

\begin{thm}\label{thm:a1chowmatrix}
	Let $\cF$ be an Ulrich sheaf on $\pp^n$ whose support is a closed subvariety $\iota\colon X\hookrightarrow\pp^n$ of dimension $k>0$ and degree $d$. Assume that the rank of $\iota^*\cF$ is one and let $\sigma\colon \cF\to D\cF$ an isomorphism. There is a symmetric $d\times d$ matrix $\Lambda$ whose entries are linear forms on $\wedge^{k+1}H^0(\pp^n,\cO_{\pp^n}(1))$ such that for all $s_0\ldots,s_k \in H^0(\pp^n,\cO_{\pp^n}(1))$ it holds that:
	\begin{enumerate}
		\item $\Lambda(s_0\wedge\cdots\wedge s_k)$ is singular if and only if the $s_j$ have a common zero on $X$.
		\item If the $s_j$ do not have a common zero on $X$ let $\pi\colon X\to\pp^k$ be the morphism defined by $s_0,\ldots,s_k$ and $U\subset\pp^k$ an open subset such that $V=\pi^{-1}(U)\subset X$ is smooth. The class of $\Lambda(s_0\wedge\cdots\wedge s_k)$ in $\GW(K)$ equals to $\deg^{\A^1}(\pi|_V)$ where $\pi|_V$ is relatively oriented by $\sigma(\pi|_V)$.
	\end{enumerate}
\end{thm}

\begin{proof}
 We first note that by \Cref{lem:symm} the isomorphism $\sigma\colon \cF\to D\cF$ is symmetric. Let $M=\Gamma_*(\cF)$ and let
 \begin{equation*}
     \mathbf{F}: \qquad  0 \longrightarrow F_c \overset{\varphi_c}{\longrightarrow} F_{c-1} \overset{\varphi_{c-1}}{\longrightarrow} \dots \overset{\varphi_2}{\longrightarrow} F_1 \overset{\varphi_1}{\longrightarrow}  F_0 \longrightarrow M \longrightarrow 0.
 \end{equation*}  
 a minimal free resolution of $M$. We choose any $S$-basis $\cB_0$ of $F_0$. The image of $\cB_0$ under the map $F_0\to M$ is sent by $\sigma$ to a generating set of $\Ext^c(M,S)(-c)$. A preimage of this generating set in the degree zero part of $\Hom_S(F_c,S)(-c)$ under the natural map $\Hom_S(F_c,S)(-c)\to\Ext^c(M,S)(-c)$ is a basis of $\Hom_S(F_c,S)(-c)$ and thus of $\Hom_S(F_c,S)$. We denote by $\cB_c$ the dual of this basis which is a basis of $F_c$. We claim that with this choice of bases of $F_0$ and $F_c$ (and any choice of basis of $F_i$ for $0<i<c$) the matrix $\gamma(\mathbf{F})$ has the desired properties. We already know that $\gamma(\mathbf{F})$ is a $d\times d$ matrix whose entries are linear forms on $\wedge^{k+1}H^0(\pp^n,\cO_{\pp^n}(1))$ which satisfies (1). It thus remains to show that $\gamma(\mathbf{F})$ is symmetric and satisfies (2). To this end let $s_0\ldots,s_k \in H^0(\pp^n,\cO_{\pp^n}(1))$ be sections that do not have a common zero on $X$. Then, as in the discussion before \Cref{thm:minres}, the morphism $\pi\colon X\to\pp^k$ defined by $s_0,\ldots,s_k$ corresponds to an inclusion of graded polynomial rings
 \begin{equation*}
R:=K[s_0,\ldots,s_k]\subset K[s_0,\ldots,s_n]=S.
\end{equation*}
 We consider the isomorphism 
 \begin{equation*}
\psi\colon M\to R^d
\end{equation*}
of $R$-modules which sends the image of $\cB_0$ under the map $F_0\to M$ to the standard basis of $R^d$. We denote by $B_i$ the representing matrix of multiplication by $s_i$ with respect to this basis and let $A_i=s_i\cdot I_d-B_i$ where $I_d$ is the $d\times d$ identity matrix. Further let $\kosz(A)$ the Koszul complex associated to the matrices $A_{k+1},\ldots,A_n$. We then have
\begin{equation}\label{eq:koszvssymm}
  \gamma(\mathbf{F})\cdot Q=\gamma(\kosz(A))
\end{equation}
where $Q$ is the representing matrix of the $R$-bilinear form on $M$ obtained by composing $\sigma\colon M\to\Ext^c(M,S)(-c)$ with the natural isomorphism from \Cref{cor:extshriek} with respect to our chosen basis. Note that $Q$ is a symmetric matrix because $\sigma$ is symmetric which further satisfies
\begin{equation}\label{eq:selfadjoint}
    QA_i=A_i^tQ
\end{equation}
for all $i=k+1,\ldots,n$ because $\sigma$ is a homomorphism of $S$-modules.
By \Cref{thm:ulrichdegreewell} and our choice of relative orientation the class of $Q$ in $\GW(K)$ equals to $\deg^{\A^1}(\pi|_V)$. Since $Q$ and $Q^{-1}$ represent the same class in $\GW(K)$ and by \Cref{eq:koszvssymm}, it now suffices to show that $\gamma(\kosz(A))\cdot Q^{-1}$ is symmetric and that $\gamma(\kosz(A))(s_0\wedge\cdots\wedge s_k)=I_d$. To this end we regard $\gamma(\kosz(A))$ as an alternating multilinear map
\begin{equation*}
    \left(H^0(\pp^n,\cO_{\pp^n}(1))^\vee\right)^c\to\Mat_d(K)
\end{equation*}
and use the explicit description 
\begin{equation}\label{eq:explicitkoszul}
    \gamma(\kosz(A))(v_1,\ldots,v_c)=\sum_{\tau \in \mathcal{S}_c} \textnormal{sgn}( \tau) \cdot A_{k+\tau(1)}(v_1)  \cdots A_{k+\tau(c)}(v_c)
\end{equation}
where $v_1,\ldots,v_c\in H^0(\pp^n,\cO_{\pp^n}(1))^\vee$, see \cite[Example~6.1.4]{KummerDiss}. \Cref{eq:explicitkoszul} together with \Cref{eq:selfadjoint} now imply
\begin{eqnarray*}
\gamma(\kosz(A))(v_1,\ldots,v_c)\cdot Q^{-1}&=&\sum_{\tau \in \mathcal{S}_c} \textnormal{sgn}( \tau) \cdot A_{k+\tau(1)}(v_1)  \cdots A_{k+\tau(c)}(v_c)\cdot Q^{-1}\\
&=&\sum_{\tau \in \mathcal{S}_c} \textnormal{sgn}( \tau) Q^{-1} \cdot A^t_{k+\tau(1)}(v_1)  \cdots A^t_{k+\tau(c)}(v_c)\\
&=&Q^{-1}\gamma(\kosz(A))^t(v_1,\ldots,v_c).
\end{eqnarray*}
This shows that $\gamma(\kosz(A))\cdot Q^{-1}$ is symmetric. Letting $s_0^\vee,\ldots,s_n^\vee$ be the dual basis of $s_0,\ldots,s_n$ we further have
\begin{eqnarray*}
\gamma(\kosz(A))(s_0\wedge\cdots\wedge s_k)&=&\gamma(\kosz(A))(s_{k+1}^\vee,\ldots,s_n^{\vee})\\
&=&\sum_{\tau \in \mathcal{S}_c} \textnormal{sgn}( \tau) \cdot A_{k+\tau(1)}(s_{k+1}^\vee)  \cdots A_{k+\tau(c)}(s_{k+c}^\vee)\\
&=&I_d
\end{eqnarray*}
since $A_{k+i}(s_{k+j}^\vee)=\delta_{ij}\cdot I_d$. This concludes the proof.
\end{proof}

\begin{ex}
 Consider the rational normal curve $\iota\colon \pp^1 \hookrightarrow \pp^n$ of degree $n$ and the symmetric Ulrich sheaf $\cF=\iota_*\left(\cO_{\pp^1}(n-1)\right)$. \Cref{thm:a1chowmatrix} gives an $n\times n$ matrix $\Lambda$ whose entries are linear forms on $\wedge^{2}K[x_0,x_1]_n$ such that for all binary forms $p,q \in K[x_0,x_1]_n$ we have:
	\begin{enumerate}
		\item $\Lambda(p\wedge q)$ has full rank if and only if $p$ and $q$ are co-prime.
		\item If $p$ and $q$ are co-prime, then the $\A^1$-degree of the map
	\begin{equation*}
	    \psi\colon\pp^1\to\pp^1,\, (x_0:x_1)\mapsto (p(x_0,x_1):q(x_0,x_1))
	\end{equation*}
	relatively oriented by an isomorphism 
	\begin{equation*}
	 \cO_{\pp^1}(n-1)\otimes\cO_{\pp^1}(n-1)  =\cO_{\pp^1}(2(n-1))\cong \omega_{\pp^1}\otimes \psi^*\omega_{\pp^1}^{-1}
	\end{equation*}
	is given by $\Lambda(p\wedge q)$.
	\end{enumerate}
  This resulting matrix $\Lambda(p\wedge q)$ is the so-called \emph{B\'ezout matrix} of $p$ and $q$. Its connection to $\A^1$-homotopy theory has already been observed in \cite{cazanave}.
\end{ex}

\begin{ex}
 Consider the elliptic curve $(E,O)$ with Weierstrass equation
 \begin{equation*}
     y^2=x^3-x
 \end{equation*}
 over a field $K$ with $\Char(K)=0$.
 Letting $P_0=(0,0)$ and $P_{-1}=(-1,0)$, the line bundle $\cL$ associated with the divisor $P_{-1}-P_0$ is $2$-torsion. We consider the embedding of $\iota\colon E\to\pp^3$ via the linear system $W$ spanned by $(s_0,s_1,s_2,s_3)=(1,x,y,x^2)$.
 The sheaf $\cF=\cO_{\pp^3}(1)\otimes\iota_*\cL$ is a symmetric Ulrich sheaf. From its free resolution we obtain the symmetric matrix
 \begin{equation*}
  \Lambda=\begin{pmatrix}
  -x_{23}+x_{12}+x_{02} & x_{23}-x_{12}-x_{02} & x_{13}+x_{03} & x_{03}+x_{01} \\ x_{23}-x_{12}-x_{02} & x_{12}+x_{02} & -x_{03} & -x_{01} \\ x_{13}+x_{03} & -x_{03} & x_{23}+x_{12} & x_{02} \\ x_{03}+x_{01} & -x_{01} & x_{02} & x_{02}
  \end{pmatrix}
 \end{equation*}
 from which we can read off the $\A^1$-degree of maps $\psi\colon E\to\pp^1$ given by two elements of $W$ and relatively oriented by an isomorphism 
	\begin{equation*}
	 \cL(1)\otimes\cL(1)  \cong \omega_{E}\otimes \psi^*\omega_{\pp^1}^{-1}.
	\end{equation*}
 Here $x_{ij}$ denotes the linear form on $\wedge^2W$ that evaluates on $s_k\wedge s_l$ to $1$ if $(i,j)=(k,l)$ and to $0$ if $\{i,j\}\neq\{k,l\}$. For instance the map defined by
 \begin{equation*}
     E\to\pp^1,\, (x,y)\mapsto (1:x^2)
 \end{equation*}
 has $x_{03}=1$ and all other Pl\"ucker coordinates zero. Thus its $\A^1$-degree equals
 \begin{equation*}
     \begin{pmatrix}
  0  & 0  & 1   & 1  \\  0  & 0  &  -1  & 0  \\  1  &  -1  & 0  & 0 \\  1  & 0  & 0 & 0
  \end{pmatrix}\cong 2\cdot(\langle1\rangle+\langle-1\rangle).
 \end{equation*}
 On the other hand, the  map defined by
 \begin{equation*}
     E\to\pp^1,\, (x,y)\mapsto (y:(2x+1)(x-2))
 \end{equation*}
 has $x_{02}=2$, $x_{12}=3$, $x_{23}=2$ and all other Pl\"ucker coordinates zero. Thus its $\A^1$-degree equals
 \begin{equation*}
     \begin{pmatrix}
  3  & -3  & 0   & 0  \\  -3  & 5  &  0  & 0  \\  0  &  0  & 5  & 2 \\  0  & 0  & 2 & 2
  \end{pmatrix}\cong 2\cdot(\langle3\rangle+\langle2\rangle).
 \end{equation*}
\end{ex}

\section{Secant varieties of curves}\label{sec:prelimcurves}
In this section we prove some cohomological statements on the secant variety that we will need later to construct Ulrich bundles. Furthermore, we define Viro frames on (the desingularization of) the secant variety.

\subsection{Symmetric products of curves}
We recall some preliminaries on symmetric products of curves, their tautological bundles and the connections to secant varieties. Unless we give another reference, the statements made in this section can be found in \cite[\S3]{secantein}. If $C$ is a curve, we will denote by $C_n$ its $n$-th symmetric product. It can be considered as the quotient $\sigma\colon C^n \to C_n$ with respect to the natural $\mathfrak{S}_n$-action. This is a non-singular projective variety of dimension $n$ that parametrizes effective divisors of degree $n$ on $C$. The addition map
\begin{equation*}
\sigma\colon C\times C_{n-1} \to C_{n}, \, (x,D) \mapsto x+D.
\end{equation*}
makes $C\times C_{n-1}$ into the universal family over $C_n$: the fiber over $D\in C_n$ is naturally isomorphic to the subscheme $D\subset C$. If $L$ is a line bundle on $C$ we can pull it back to the universal family $C\times C_{n-1}$ and then pushforward to define the tautological bundle $E_{n,L} = \sigma_*pr_{C}^*L$. The sheaf $E_{n,L}$ is a vector bundle of rank $n$  on $C_n$ whose fiber at $D$ can identified with $H^0(C,L\otimes \mathcal{O}_{D})$. Here $\cO_{D}$ denotes the structure sheaf of $D$ considered as subscheme of $C$.
\begin{Def}
 A line bundle $L$ on $C$ is called \emph{$k$-very ample} if 
 \begin{equation*}
  h^0(C,L\otimes \mathcal{O}_C(-D))=h^0(C,L)-(k+1)   
 \end{equation*}
 for all effective divisors $D$ of degree $k+1$.
\end{Def}

\begin{rem}
A line bundle $L$ is $0$-very ample if and only if it is globally generated, and it is $1$-very ample if and only if it is very ample. If $k>1$, then $L$ is $k$-very ample if and only it induces an embedding $\varphi_L\colon C\hookrightarrow \mathbb{P}^n$ such that any, possibly coincident, $k+1$ points on $C$ span a $k$-plane in $\mathbb{P}^r$. 
\end{rem}

\begin{rem}\label{rem:identifiability3va}
A line bundle $L$ is $3$-very ample if and only if its complete linear system induces an embedding $\varphi_L\colon C \hookrightarrow \pp^n$ such that the secant variety $\Sigma = \Sigma(C,L)$ of $\varphi_L(C)$ is identifiable. This means that any point in $\Sigma\setminus \varphi_L(C)$ is contained in a unique secant line. Indeed, two secant lines $\ell(p_1,p_2)$ and $\ell(p_3,p_4)$, with $p_i\in C$ meet, if and only if the points $p_1,\dots,p_4$ span a $2$-plane. 
\end{rem}

We have $H^0(C,L)= H^0(C_n,E_{n,L})$. Further, if the line bundle $L$ is $(n-1)$-very ample, then $E_{n,L}$ is globally generated. Thus in this case there is a surjective map
\begin{equation*}
 \operatorname{ev}_L\colon H^0(C,L)\otimes \mathcal{O}_{C_n} {\to} E_L
\end{equation*}
 of vector bundles on $C_n$ which is an isomorphism on global sections.
 
\subsection{Line bundles on the symmetric product}
We now define some line bundles on $C_n$ and describe their relations to each other.
More precisely, to any line bundle $L$ on $C$ we can associate two line bundles on $C_n$. Firstly, the line bundle $N_{n,L}$ on $C_n$ is defined as the determinant of $E_L$. We denote  $\mathcal{O}_{C_n}(\delta) = N_{\cO_C}^{\vee}$ and $\delta$ is the corresponding divisor class. Further, on the direct product $C^n$ we have the line bundle $L^{\boxtimes \, n} = \pr_1^* L \otimes \dots \otimes \pr_n^*L$. Since $L^{\boxtimes \, n}$ is invariant under permuting the components, it descends to a line bundle $S_{n,L}$ on $C_n$ satisfying $\sigma^*S_L \cong L^{ \boxtimes n}$ and the induced map
\begin{equation*}
\Pic(C) \to \Pic(C_n),\qquad L \mapsto S_{n,L}
\end{equation*}
is a group homomorphism. For a closed point $P\in C$ we have the divisor $S_P = P+C_{n-1} \subset C_n$ and the associated line bundle is precisely $S_{\mathcal{O}_C(P)}$.  Extending this linearly, we can define the divisor $S_D$ on $C_n$ for every divisor $D$ on $C_n$. The associated line bundle of $S_D$ is $S_{\mathcal{O}_C(D)}$ and we will use both notations interchangeably.
Finally, a distinguished divisor on $C_n$ is the locus $\Delta\subset C_n$ of all non-reduced divisors, and for any line bundle $L$ on $C$ we set $A_{n,L} = S_{n,L}(-\Delta)$ as in \cite{blowupenp}. The following lemma summarizes the relations of these line bundles and divisors to each other.
\begin{lem}\label{lem:lincalc}
Letting $\mathcal{O}_{C_n}(\delta) =  \det E_{n,\cO_C}^{\vee}$ we have:
\begin{enumerate}[a)]
    \item $\mathcal{O}_{C_n}(2\delta) \cong \mathcal{O}_{C_n}(\Delta)$;
    \item $N_{n,L}\otimes \mathcal{O}_{C_n}(\delta) \cong S_{n,L}$ for every line bundle $L$ on $C$;
    \item The cotangent vector bundle is $\Omega^1_{C_n} \cong E_{n,\omega_n}$ and the canonical line bundle on $C_n$ is $\omega_{C_n} \cong N_{n,\omega_C}$.
\end{enumerate}
\end{lem}

The cohomology of these line bundles is known:

\begin{lem}\label{lemma:cohomologysymmetricproduct} 
	Let $L$ and $M$ be line bundles on $C$.  Then we have isomorphisms
	\begin{eqnarray*}
	    H^i(C_n,N_{n,L})& \cong& \wedge^{n-i}H^0(C,L)\otimes \operatorname{Sym}^i H^1(C,L),\\ 
        H^i(C_n,S_{n,L})&\cong& \operatorname{Sym}^{n-i} H^0(C,L) \otimes \wedge^{i}H^1(C,L), \\
     H^i(C_n,E_{n,M}\otimes S_{n,L})& \cong& H^0(C,L\otimes M) \otimes \operatorname{Sym}^{n-1-i}H^0(C,L) \otimes \wedge^{i}H^1(C,L) \\
     &\oplus& H^1(C,L\otimes M) \otimes \operatorname{Sym}^{n-i}H^0(C,L) \otimes \wedge^{i-1}H^1(C,L)
	\end{eqnarray*}
\end{lem}
\begin{proof}
	The first two follow from the discussion after \cite[Proposition 6.3]{KrugRemarks2016}. For the last one, we can use the same strategy as in \cite[Lemma 4.4]{AgoGL} use the universal family $\sigma\colon C_{n-1} \times C \to C_{n}$. By definition of $E_{n,L}$ and by the projection formula $E_{n,M} \otimes S_{n,L}\cong \sigma_*(pr_C^*M\otimes \sigma^* S_{n,L})$. Since the map $\sigma$ is finite, it also follows that $H^i(C_n,E_{n,M}\otimes S_{n,L}) \cong H^i(C\times C_{n-1},pr_C^*L \otimes \sigma^*S_{n,M})$. To conclude, it is enough to observe that $\sigma^* S_{n,M} \cong pr_C^*(M)\otimes pr_{C_{n-1}}^*S_{n-1,M}$ and use the formulas for $H^i(C_{n-1},S_{n-1,M})$.
\end{proof}

\subsection{Secants and Viro frames}\label{sec:projbundles}
Let $L$ be a 3-very ample line bundle on the curve $C$ and $\iota\colon C\hookrightarrow\pp(H^0(C,L)^\vee)=\pp^r$ the corresponding embedding. The fact that $L$ is $3$-very ample means precisely that no four (possibly coincident) points on $C$ are contained in a plane in $\mathbb{P}^r$. In this section we recall some observations on the projective space bundle $B=\pp(E_{2,L})$ where $E_L$ is the tautological bundle on $C_2$ defined in \Cref{sec:prelimcurves}. Since $E_L$ is globally generated, we have a natural identification \begin{equation*}B=\{(p,x+y)\in \pp^r\times C_2\mid p\in\ell(x+y)\}\end{equation*}where $\ell(x+y)\subset\pp^r$ denotes the line in $\mathbb{P}^r$ spanned by the subscheme $x+y$ of $C$. We consider the projections
\begin{equation*}
p\colon B \to C_2, \, \varepsilon\colon B \to \Sigma
\end{equation*}where $\Sigma$ is the secant variety of $\iota(C)\subset\pp^r$.
 If $L$ is $3$-very ample, then the map $\varepsilon$ is an isomorphism outside of $C\subsetneq \Sigma$ and we denote the preimage of $C$ by $E$. In this case we can identify $E$ with $C\times C$ via
\begin{equation*}
\label{eq:exceptionaldivisorn}
C\times C \to B, \, (x,y) \mapsto (\iota(x),x+y)
\end{equation*}
In particular, the fiber of $\varepsilon$ over $x\in C$ is naturally identified with $\{x\}\times C$. Now we would like to compute the class of the divisor $E$, but instead of working with classes, we will introduce some explicit rational functions, that will be useful later on.

Let $z_0,z_1$ be two linearly independent global sections of $\cO_{\pp^r}(1)$. Let $H_\infty\subset\pp^r$ be the hyperplane $H_{\infty}=\{z_0=0\}$. This pulls back to the effective divisors $D=\iota^*(H_\infty)$ on $C$ and $H=\varepsilon^*(H_\infty)$ on $B$. Note that $L\cong\cO_C(D)$ and $\cO_B(1)\cong\cO_B(H)$. Consider the rational functions $q=\frac{z_1}{z_0}$ on $\pp^r$, $f=\iota^*q$ on $C$ and $h=\varepsilon^*q$ on $B$. We denote $f_1=f\otimes 1$ and $f_2=1\otimes f$, which are rational functions on $C\times C$. Further we let $g_1=f_1+f_2$ and $g_2=f_1\cdot f_2$. These are rational functions on $C_2$, since they are $\mathfrak{S}_2$-invariant. Finally, we denote by $\widetilde{\Delta}$ the closure in $C_2$ of the set
\begin{equation*}
    \{x+y\mid x,y\in (C\setminus D),\, x\neq y, \, f(x)=f(y)\}.
\end{equation*}

\begin{lem}\label{lem:diffcomp1}
 Consider the rational function $G=p^*g_2-p^*g_1\cdot h+h^2$ on $B$. Its divisor is \begin{equation*}\operatorname{div}(G)=E+p^*(2\widetilde{\Delta}-S_{D})-2H.\end{equation*}
\end{lem}

\begin{proof}
A point on $B$ has the form $(z,x+y)$ where $z$ is a point on the line spanned by $\iota(x)$ and $\iota(y)$. Thus $G$ evaluated at this point is
 \begin{align}
 \label{eq:Gexpress}
 G&= f_1(x)\cdot f_1(y)-(f_1(x)+f_2(y))\cdot h(z)+h(z)^2 \\
 &=(f_1(x)-h(z))(f_2(y)-h(z)) \\
 &=(q(\iota(x))-q(z))(q(\iota(y))-q(z)).   
 \end{align}
 Since $z$ is on the line spanned by $\iota(x)$ and $\iota(y)$, we find that $f(x)=f(y)$ implies $q(z)=q(\iota(x))=q(\iota(y))$.
 This means that $G$ vanishes of order $2$ along $p^*(\widetilde{\Delta})$. If instead $f(x)\neq f(y)$, then $G$ vanishes on $(z,x+y)$ if and only if $z=\iota(x)$ or $z=\iota(y)$ which means that $(z,x+y)$ lies on $E$. This shows that the zero divisor of $G$ is $E+p^*(2\widetilde{\Delta})$. Finally, $G$ has a pole of order $2$ whenever $z$ lies on $H_\infty$ and an ordinary pole whenever $\iota(x)$ or $\iota(y)$ lies on $H_\infty$. This shows that the pole divisor of $G$ is $p^*(S_{D})+2H$.
\end{proof}

\begin{lem}\label{lem:diffcomp2}
 The divisor of the rational differential $2$-form $dg_1\wedge dg_2$ on $C_2$ is $S_{\divv(df)}-S_{D}+\widetilde{\Delta}$.
\end{lem}

\begin{proof}
 One computes that 
 \begin{equation*}\label{eq:g1g2}
 dg_1\wedge dg_2=(f_1-f_2)\cdot(d f_1 \wedge d f_2).    
 \end{equation*}
 The divisor of this rational differential $2$-form on $C\times C$ is $\Delta_C+\sigma^*(S_{\divv(df)}-S_{D}+\widetilde{\Delta})$ where $\Delta_C\subset C\times C$ is the diagonal. Now the claim follows from the fact that $\Delta_C$ is the ramification divisor of $\sigma$.
\end{proof}

Now, fix some rational differential $\omega$ on $C$. Then there is a rational function $t$ on $C$ such that $\omega=t\cdot df$. Let $t_1=t\otimes 1$, $t_2=1\otimes t$ and consider the rational function $F=t_1t_2$ on $C_2$. Define at this point the following rational 3-form on $B$ 

\begin{equation}\label{eq:Omega}
     \Omega(q,\omega):=p^*\left(\frac{F}{G}\cdot(dg_2 \wedge dg_1)\right)\wedge dh
 \end{equation}

 \begin{lem}
 The divisor of the form $\Omega(q,\omega)$ is
\begin{equation}\label{eq:concretediff}
 \divv \Omega(q,\omega) = p^*(S_{\divv(\omega)})-E.  
 \end{equation} 
 \end{lem}
\begin{proof}
By \Cref{lem:diffcomp2}  the divisor of the rational differential $2$-form $F\cdot(dg_1\wedge dg_2)$ on $C_2$ is $S_{\divv(\omega)}-S_{D}+\widetilde{\Delta}$. If the differential $dh$ is defined and nonzero at a point $(z,x+y)$ on $B$, then it is in the span of $dg_1$ and $dg_2$ at this point if and only if $f(x)=f(y)$. Thus the divisor of the differential $3$-form $F\cdot(dh\wedge dg_1\wedge dg_2)$ is 
\begin{equation*}
p^*(\divv(F\cdot(dg_1\wedge dg_2)))+ p^*(\widetilde{\Delta})-2{H}= p^*(S_{\divv(\omega)}-S_{D}+2\widetilde{\Delta})-2{H}. 
\end{equation*}
Combining this with \Cref{lem:diffcomp1}, we conclude.
\end{proof}

\begin{rem}\label{cor:canonicalBn}
In particular, this computes one canonical divisor of $B$ as $K_B = p^*S_{K_C}-E$. Using the formula $K_B \sim p^*(K_{C_2})+(p^*N_{L})-2H$ for the canonical class of the projective bundle $p\colon B\to C_2$, we can also compute that $E\sim 2H-p^*(S_D-\Delta)$ where $D$ is any divisor such that $\mathcal{O}_C(D)\cong L$.
\end{rem}

The differential $3$-form $\Omega(q,\omega)$
 depends a priori on the choice of the rational function $q=\frac{z_1}{z_0}$ on $\pp^r$ and on the differential $\omega$ on $C$. Letting $\omega'=t'\cdot\omega$ another differential, we get that
 \begin{equation}\label{eq:depndonomega}
     \Omega(q,\omega')=p^*(t'\otimes t')\cdot\Omega(q,\omega).
 \end{equation}
 where we see $t'\otimes t'$ as a rational function on $C_2$. 
\Cref{eq:depndonomega} shows in particular:
\begin{cor}\label{cor:deponomega}
 If $\omega'$ differs from $\omega$ only by a scalar or by the square of a rational function, then $\Omega(q,\omega')$ differs from $\Omega(q,\omega)$ only by a square.
\end{cor}
Next we will show that $\Omega(q,\omega)$ does not depend on $q$.
\Cref{eq:concretediff} shows that the divisor of $\Omega(q,\omega)$ does not depend on $q$. Therefore, it suffices to evaluate $\Omega(q,\omega)$ at one basis of the tangent space of $B$ at one point and show that the result does not depend on the choice of $q$. To this end let $a,b$ be two distinct points of $C$ where $\omega$ does not have a zero or a pole, and let $z\in \ell(a,b)\setminus \{ a,b \}$ be a point on the line spanned by them. Set then $Q=(p,a+b)$. Since $a\neq b$, the differential  
\begin{equation*}
    \T_{(a,b)}(\sigma)\colon \T_{a}C\times \T_bC\to \T_{a+b}C_2
\end{equation*}
is an isomorphism. Further, let $F=\ell(a+b)$ be the fiber of $p\colon B\to C_2$ over $a+b$.  There is a unique isomorphism $\psi\colon\pp^1\to F$ that satisfies $\psi(1:0)=\iota(a)$, $\psi(1:1)=z$ and $\psi(0:1)=\iota(b)$. Via the differentials $\T_{(a,b)}(\sigma)$ and $\T_{(1:1)}(\psi)$ we obtain the short exact sequence of tangent spaces
\begin{equation*}
    0\to\T_{(1:1)}\pp^1\to\T_QB\to \T_{a}C\times \T_bC\to 0
\end{equation*}
whose determinant gives the isomorphism
\begin{equation*}
    \Psi\colon\T_{a}C\otimes\T_{(1:1)}\pp^1\otimes\T_bC \to\wedge^3\T_QB.
\end{equation*}
Let $v\in \T_aC$ and $w\in\T_bC$ vectors on which $\omega$ evaluates to one. Let $x_0,x_1$ be the homogeneous coordinates on $\pp^1$, let $t=\frac{x_1}{x_0}$ and $u\in\T_{(1:1)}$ be the vector on which $dt$ evaluates to one. Now we consider
\begin{equation*}\label{eq:viroframe}
    \Psi(v\otimes u\otimes w)\in\wedge^3\T_QB.
\end{equation*}
Besides the choice of the differential $\omega$, the construction depends a priori on the choice of the preimage $(a,b)$ of $a+b$ under $\sigma$. However, choosing the other preimage, we obtain
\begin{equation*}
    \Psi(w\otimes (-u)\otimes v)=\Psi(v\otimes u\otimes w).
\end{equation*}
Thus $\Psi(v\otimes u\otimes w)$ is independent of this choice. 
\begin{Def}\label{def:viroframe}
 We call 
 \begin{equation*}
  \V(\omega,Q):= \Psi(v\otimes u\otimes w)\in\wedge^3\T_QB   
 \end{equation*}
  the
\emph{$\omega$-Viro frame} at $Q$.
\end{Def}
 Now we evaluate $\Omega(q,\omega)$ at $\V(\omega,Q)$. We can write
\begin{equation*}
    s:=q\circ\psi=\frac{\alpha t+\beta}{\gamma t+\delta}
\end{equation*}
for some scalars $\alpha,\beta,\gamma$ and $\delta$. Using \Cref{eq:Gexpress} and (\ref{eq:g1g2}) we calculate that the evaluation of $\Omega(q,\omega)$ at $\V(\omega,Q)$ equals to the evaluation of
\begin{equation*}
    \frac{q(\iota(a))-q(\iota(b))}{(q(\iota(a))-q(p))(q(\iota(b))-q(p))}ds=\frac{s(1:0)-s(0:1)}{(s(1:0)-s(1:1))(s(0:1)-s(1:1))}ds
\end{equation*}
at $u$. On the one hand we have
\begin{equation*}
    \frac{s(1:0)-s(0:1)}{(s(1:0)-s(1:1))(s(:1)-s(1:1))}=\frac{(\gamma+\delta)^2}{\alpha\delta-\beta\gamma}
\end{equation*}
and on the other hand one calculates
\begin{equation*}
    ds=\frac{\alpha(\gamma t+\delta)-\gamma(\alpha t+\beta)}{(\gamma t+\delta)^2}dt
\end{equation*}
which evaluates at $u$ to $\frac{\alpha\delta-\beta\gamma}{(\gamma+\delta)^2}$. Thus we have proven the following.

\begin{thm}\label{thm:viroframe}
 The differential $3$-form $\Omega(\omega):=\Omega(q,\omega)$ does not depend on the choice of $q$. At every $\omega$-Viro frame it evaluates to $1$.
\end{thm}

\subsection{Some cohomology} Now we prove some technical cohomological vanishing statements that we will use to construct an Ulrich sheaf on the secant variety.

Suppose again that $\iota\colon C\hookrightarrow\pp(H^0(C,L)^\vee)=\pp^r$ is an embedding by a $3$-very ample line bundle let $\Sigma$ be the secant variety and 
\[ C_2 \overset{p}{\longleftarrow} B \overset{\varepsilon}{\longrightarrow} \Sigma \] 
as before. We take a vector bundle $F$ on $C_2$ and we consider the induced sheaf $\varepsilon_*p^* G$ 
on $\Sigma$. Our lemma allows to transfer the cohomology of this sheaf from $\Sigma$ to $B$:

\begin{lem}\label{lem:condition1}
Assume that $H^1(\{x\}\times C , {\sigma^*F}_{|\{x\}\times C}) = 0$ for all $x\in C$, or equivalently that $H^1(x+C,F_{|x+C}) = 0$ for all $x\in C$. Then 
\[ H^j(\Sigma,\varepsilon_*p^*F\otimes\mathcal{O}_{\pp^n}(i)) \cong H^j(B,p^*F\otimes \mathcal{O}_B(i)) \qquad \text{ for all } i,j. \]
\end{lem}
\begin{proof}
Assume that $R^1\varepsilon_*p^*F = 0$. Since $\varepsilon\colon B \to \Sigma$ has fibers of dimension at most one, the Leray spectral sequence, together with the projection formula shows that $H^i(B,p^*F\otimes \mathcal{O}_B(s) ) \cong H^i(\Sigma,\varepsilon_*p^*F\otimes \mathcal{O}_{\Sigma}(q))$ for all $i\geq 0$. To prove that $R^1\varepsilon_*p^*F = 0$ we follow an argument of \cite{UlleryNormalitySecants2016} and \cite{secantein} that we repeat here. Consider the fiber $\varepsilon^{-1}(x)$  over $x\in  C$, and let $\mathcal{I}=\mathcal{I}_{\varepsilon^{-1}(x)/B}$ be its sheaf of ideals on $B$. Then the Theorem of Formal Functions proves that the vanishing $R^1\varepsilon_*p^*F$ follows from  $H^1(B,p^*F\otimes \mathcal{O}_B/\mathcal{I}^k) = 0$ for all $k\geq 1$.
Looking at the sequences
\[ 0 \to \mathcal{I}^{k}/\mathcal{I}^{k+1} \to \mathcal{O}_B/\mathcal{I}^{k+1} \to \mathcal{O}_B/\mathcal{I}^k \to 0 \]
it is enough to show that  $H^1(B,p^*F\otimes \mathcal{I}^{k}/\mathcal{I}^{k+1})=0$ for all $k\geq 0$. Now, we know that in our case $\varepsilon^{-1}(x)$ is a smooth subvariety of the smooth variety $B$, hence it is a locally complete intersection, so that the sheaf $\mathcal{I}^{k}/\mathcal{I}^{k+1}$ is naturally isomorphic to the symmetric power of the conormal bundle 
$\operatorname{Sym}^k(\mathcal{I}/\mathcal{I}^2) \cong \mathcal{I}^{k}/\mathcal{I}^{k+1}$. 
In summary, we are left to prove the vanishings
\[ H^1(\varepsilon^{-1}(x),(p^*F|_{\varepsilon^{-1}(x)}\otimes \operatorname{Sym}^k(\mathcal{I}/\mathcal{I}^2))=0 \qquad \text{ for all } k\geq 0 \]
Now we identify the pieces in our situation: we know from before that the fiber is identified with $\{x\}\times C \subset C\times C$, and  \cite[Lemma 2.3]{UlleryNormalitySecants2016} shows that
that $\mathcal{I}/\mathcal{I}^2 \cong \mathcal{O}_C \oplus L(-2x)$, hence $\operatorname{Sym}^k(\mathcal{I}/\mathcal{I}^2) \cong \bigoplus_{h=0}^k (L(-2x))^{\otimes h}$.  Since we are assuming that $L$ is $3$-very ample, it must be that $h^0(C,L)\geq 3$, so that $L(-2x)$ is effective. Thus, all the vanishings that we need are implied by the single vanishing
\[ H^1(\varepsilon^{-1}(x),(p^*F|_{\varepsilon^{-1}(x)})=0   \]
With the identification of $\varepsilon^{-1}(x)$ with $\{x\}\times C$, this is exactly the vanishing in our assumption. 
\end{proof}
Now, we want to look for an Ulrich sheaf of the form
\[  \mathcal{F} = \varepsilon_*p^*F \otimes \mathcal{O}_{\pp^n}(t) \]
for a certain $t\in\mathbb{Z}$.

\begin{lem}\label{lem:ulrichconditions}
Let $F$ be a vector bundle on $C_2$ and let $t\in\mathbb{Z}$ be such that
\begin{enumerate}
\item $H^1(x+C,F_{|x+C}) = 0$  for all $x\in C$.
\item $H^{i}(C_2,F\otimes S^{t-i}E_{2,L}) = H^{i-1}(C_2,F\otimes S^{t-i}E_{2,L})=0$ for $1\leq i \leq 3$ and $i\leq t$.
\item $H^{i}(C_2,F^{\vee}\otimes S^{2-i-t}E_{2,L}\otimes A_{2,\omega_C\otimes L})=0$ and $H^{i-1}(C_2,F^{\vee}\otimes S^{2-i-t}E_{2,L}\otimes A_{2,\omega_C\otimes L}) = 0$ for $1\leq i\leq 3$ and $i\leq 2-t$.
\end{enumerate}
Then $\mathcal{F}=\varepsilon_*p^*F \otimes \mathcal{O}_{\pp^n}(t)$ is an Ulrich sheaf on $\Sigma$ of rank equal to the rank of $F$.
\end{lem} \begin{proof}
Condition (1) and \Cref{lem:condition1} show that the conditions for being Ulrich are equivalent to
\[ H^i(B,p^*F\otimes \mathcal{O}_B(t-i)) = H^{i-1}(B,p^*F\otimes \mathcal{O}_B(t-i))=0 \qquad \text{ for } 1\leq i \leq 3 \]
We compute these via the Leray spectral sequence for $p\colon B \to C_2$.  The projection formula gives that $R^ap_*(p^*F\otimes \mathcal{O}_B(t-i)) = F\otimes R^ap_*\mathcal{O}_B(t-i)$ for all $a$, and \cite[Exercise III.8.4]{Hart77} shows that
\[ 
R^ap_*\mathcal{O}_B(t-i) \cong 
\begin{cases} 
S^{t-i}E_{2,L}, &\text{ if } a=0 \text{ and } i\leq t \\
S^{i-t-2}E^{\vee}_{2,L}\otimes N_{2,L}^{\vee},  &\text{ if } a=1 \text{ and } i\geq t+2 \\
0, &\text{ otherwise. }
\end{cases}\]
and then the Leray spectral sequence shows that the conditions for being Ulrich are equivalent to the following vanishings for $1\leq i\leq 3$: 
\begin{align*}
&H^i(C_2,F\otimes S^{t-i}E_{2,L}) = H^{i-1}(C_2,F\otimes S^{t-i}E_{2,L})=0, &  i\leq t \\
&H^{i-1}(C_2,F\otimes S^{i-t-2}E^{\vee}_{2,L}\otimes N_{2,L}^{\vee}) \\=& H^{i-2}(C_2,F\otimes S^{i-t-2}E^{\vee}_{2,L}\otimes N_{2,L}^{\vee})=0, &   i\geq t+2 
\end{align*}
and to get the conditions that we are looking for, we can apply Serre duality, together with the observation that $\omega_{C_2} \cong N_{2,\omega_C}$.
\end{proof}
We can also control the effect of duality 

\begin{lem}\label{lem:duality}
Let $F$ be a vector bundle on $C_2$  such that $H^1(C+x,F) = H^1(C+x,F^{\vee}\otimes A_{2,L\otimes \omega_C}) = 0$ for all $x\in C$ and let also $\mathcal{F} \cong \varepsilon_*p^*F \otimes \mathcal{O}_{\pp^r}(t)$ for one $t\in \mathbb{Z}$. Then
\[ D(\varepsilon_*p^*F\otimes \mathcal{O}_{\mathbb{P}^r}(t)) \cong \varepsilon_*p^*(F^{\vee}\otimes A_{2,L\otimes \omega_C}) \otimes \mathcal{O}_{\mathbb{P}^r}(2-t). \]
\end{lem}
\begin{proof}
The hypothesis implies that $R^1\varepsilon_*(p^*F) = R^1\varepsilon_*(p^*F^{\vee}\otimes A_{2,\omega_C\otimes L}) = 0$ as in the proof of \Cref{lem:condition1}. Now we use Grothendieck duality:  define $f\colon B\to \mathbb{P}^r$ as the composition $B\to\Sigma \hookrightarrow \mathbb{P}^r$. If $\mathcal{G}$ is a vector  bundle on $B$, then
\[ D(\varepsilon_*\mathcal{G}) = \mathcal{E}xt^{r-3}_{\mathbb{P}^r}(f_*\mathcal{G},\mathcal{O}_{\mathbb{P}^r}(3-r))  \]
by definition. In particular, assume that $R^1f_*\mathcal{G}=0$, so that $f_*\mathcal{G} \cong Rf_*\mathcal{G}$ in the derived category $D^b(\mathbb{P}^n)$. Then Grothendieck duality \cite[Theorem 3.34]{HuyFM} gives an isomorphism in $D^b(\mathbb{P}^n)$:
\begin{align*} 
R\mathcal{H}om_{\mathbb{P}^r}\left( f_*\mathcal{G},\mathcal{O}_{\mathbb{P}^r}(3-r)\right)[r-3] &\cong Rf_*R\mathcal{H}om_B(\mathcal{G},f^*\mathcal{O}_{\mathbb{P}^r}(3-r)\otimes  \omega_B\otimes f^*\omega_{\mathbb{P}^r}^{\vee}) \\
& \cong Rf_*R\mathcal{H}om_B(\mathcal{G},\omega_B\otimes \mathcal{O}_B(4)) \\
& \cong Rf_*R\mathcal{H}om_B(\mathcal{G},p^*(S_{L\otimes \omega_C}(-\Delta))\otimes \mathcal{O}_B(2))\\
& \cong Rf_*(\mathcal{G}^{\vee}\otimes p^*(S_{L\otimes \omega_C}(-\Delta))\otimes \mathcal{O}_B(2)) 
\end{align*}
where we used the fact in \Cref{cor:canonicalBn} that $\omega_B\cong p^*(S_{L\otimes \omega_C}(-\Delta))\otimes \mathcal{O}_B(2)$.  Since $R^1{\varepsilon}_*p^*F = 0$, we can apply this reasoning to $\mathcal{G}=p^*F$, and we obtain
\begin{align*} 
R\mathcal{H}om(f_*p^*F,\mathcal{O}_{\mathbb{P}^r}(3-r))[r-3] &\cong Rf_*(p^*(F^{\vee}\otimes A_{2,L\otimes \omega_C})) \otimes \mathcal{O}_{\mathbb{P}^r}(2) \\
&\cong \varepsilon_*(p^*(F^{\vee}\otimes A_{2,L\otimes \omega_C})) \otimes \mathcal{O}_{\mathbb{P}^r}(2) 
\end{align*}
where the last isomorphism comes from the fact that $R^1\varepsilon_*(p^*(F^{\vee}\otimes A_{2,L\otimes \omega_C})) = 0$. If we tensor both sides by $\mathcal{O}_{\mathbb{P}^r}(-t)$, we get what we want.
\end{proof}

Now we can finally construct Ulrich sheaves on the secant variety:

\begin{thm}\label{thm:ulrichsecant}
	Let $\alpha$ be a line bundle on $C$ such that $H^0(C,\alpha)=H^1(C,\alpha)=0$. Then the coherent sheaf \begin{equation*}
	\mathcal{F} = \varepsilon_*p^*A_{2,L\otimes \alpha}
	\end{equation*}
is an Ulrich sheaf of rank one on $\Sigma$. If furthermore $\alpha$ is a theta characteristic, meaning that $\alpha\otimes\alpha\cong \omega_C$, then this is a symmetric Ulrich sheaf of rank one on $\Sigma$.
\end{thm}

\begin{proof}
To prove that the sheaf is Ulrich, we check the conditions of \Cref{lem:ulrichconditions} for $F=A_{2,L\otimes \alpha}$. The first condition $H^1(x+C,{A_{2,L\otimes \alpha}}_{|x+C})=0$ for each $x\in C$ becomes $H^1(C,\alpha\otimes L(-2x))=0$ for all $x\in C$. Since $h^0(C,L)\geq 3$ we see that $L(-2x)$ is effective, and since $H^1(C,\alpha)=0$, it follows that $H^1(C,\alpha\otimes L(-2x))=0$ as well.  Observe that in the notation of \Cref{lem:ulrichconditions}, we have $t=0$, so that the second condition is empty. For the third condition we have to check some cohomology vanishings, which in our case are :
\begin{align*}
H^{0}(C_2,S_{2,\omega_C\otimes \alpha^{\vee}}\otimes E_{2,L}) &= H^{1}(C_2,S_{2,\omega_C\otimes \alpha^{\vee}}\otimes E_{2,L}) = 0. \\
H^{1}(C_2,S_{2,\omega_C\otimes \alpha^{\vee}}) &= H^{2}(C_2,S_{2,\omega_C\otimes \alpha^{\vee}}) = 0
\end{align*}
By Serre duality, we know that $H^0(C,\omega_C\otimes \alpha^{\vee})=H^1(C,\omega_C\otimes \alpha^{\vee})=0$ and then the above vanishing follow from the formulas in \Cref{lemma:cohomologysymmetricproduct}. This proves that $\mathcal{F}$ is Ulrich. We can also compute $D(\mathcal{F})$ using \Cref{lem:duality}: the other condition that we need to check is that $H^1(x+C,A_{2,L\otimes (\omega_C\otimes \alpha^{\vee})}) = 0$ for all $x\in C$, but this can be proven as in the beginning of this proof. Then \Cref{lem:duality} shows that
\[ D(\mathcal{F}) \cong \varepsilon_*p^*S_{2,\omega_C\otimes \alpha^{\vee}}\otimes \mathcal{O}_{\mathbb{P}^r}(2)\]
To conclude, we will show that this is isomorphic to $\varepsilon_*p^*(A_{2,\omega_c\otimes \alpha^{\vee}})$, so that if $\alpha$ is a theta characteristic, we are done thanks to Lemma \ref{lem:symm}. We see that
\[ \varepsilon_*p^*S_{2,\omega_C\otimes \alpha^{\vee}}\otimes \mathcal{O}_{\mathbb{P}^r}(2) \cong \varepsilon_*(p^*(S_{\omega\otimes \alpha^{\vee}})\otimes \mathcal{O}_B(2)) \cong \varepsilon_*(p^*A_{2,L\otimes \omega_C\otimes \alpha^{\vee}}\otimes \mathcal{O}_B(E)) \]
where the last isomorphism comes from \ref{cor:canonicalBn}. Now consider the exact sequence on $B$ 
\begin{equation}\label{eq:derivedsequence}  
0 \to p^*A_{2,L\otimes \omega_C\otimes \alpha^{\vee}} \to p^*A_{L\otimes \omega_C\otimes \alpha^{\vee}}\otimes \mathcal{O}_B(E) \to p^*A_{L\otimes \omega_C\otimes \alpha^{\vee}} \otimes \mathcal{O}_E(E) \to 0  
\end{equation}
and observe that under the isomorphism $C\times C\to E$ given by $(x,y)\mapsto (\iota(x),x+y)$ it holds that
\[
p^*A_{L\otimes \omega_C\otimes \alpha^{\vee}} \otimes \mathcal{O}_E(E) \cong p^*(S_{\omega_C\otimes \alpha^{\vee}})\otimes \mathcal{O}_E(2) \cong pr_1^*(L^{\otimes 2}\otimes \omega_C\otimes \alpha^{\vee}) \otimes pr_2^*(\omega_C\otimes \alpha^{\vee})
\]
and since the composition $C\times C \to E \overset{f}{\to} \mathbb{P}^r$ is identified with the first projection $C\times C\to C$ we see that
\begin{align*}
Rf_*(p^*A_{L\otimes \omega_C\otimes \alpha^{\vee}} \otimes \mathcal{O}_E(E)) 
&\cong R{pr_1}_*(pr_1^*(L^{\otimes 2}\otimes \omega_C\otimes \alpha^{\vee}) \otimes pr_2^*(\omega_C\otimes \alpha^{\vee}))\\
&\cong (L^{\otimes 2}\otimes \omega_C\otimes \alpha^{\vee})\otimes R{pr_1}_* pr_2^*(\omega_C\otimes \alpha^{\vee}) = 0
\end{align*}
where the last equality comes from Grauert's theorem, together with the fact that $\omega_C\otimes \alpha^{\vee}$ has no cohomology. At this point \Cref{eq:derivedsequence} shows that $\varepsilon_*(p^*A_{L\otimes \alpha'}\otimes \mathcal{O}(E)) \cong  \varepsilon_*(p^*A_{L\otimes \alpha'})$.
\end{proof}
\begin{rem}
Looking at the proofs in this section, we see that we do  not need the full assumption of $3$-very ampleness. What is actually needed is that $R^1\varepsilon_*p^*(A_{L\otimes \alpha})=0$ along the map $\varepsilon\colon B \to \Sigma$. For example, the proof of \Cref{lem:condition1} shows that this holds whenever the map $\varepsilon\colon B \to \Sigma$ has only zero-dimensional fibers over $\Sigma\setminus C$. Geometrically, this means that there are no infinitely secant lines passing through a point $x\in \mathbb{P}^r\setminus C$.
\end{rem}

\subsection{Rational normal curves} 

In the case of rational normal curves, we can also find Ulrich sheaves of higher rank. Thus, we let now $C=\pp^1$ and $L=\mathcal{O}_{\pp^1}(n)$, with $n\geq 3$. Recall that in this case there is a natural identification of the second symmetric product $(\pp^1)_2$ with a projective space $\pp^2$, so that the divisors $x+C \subset C_2$ are lines in $\pp^2$ and the locus of non-reduced divisors $\Delta\subset (\pp^1)_2$ is a smooth conic. In particular, we see that $S_{2,\mathcal{O}_{\pp^1}(d)} \cong \mathcal{O}_{\pp^2}(d)$ and $\mathcal{O}_{(\pp^1)_2}(\delta)\cong \mathcal{O}_{\pp^2}(1)$. We also recall that the quotient bundle $Q$ on $\pp^2$ is defined by the Euler exact sequence
\begin{equation}\label{eq:exseqQ} 
0 \to \mathcal{O}_{\pp^2}(-1) \to \mathcal{O}_{\pp^2}^{\oplus 3} \to Q \to 0 
\end{equation}
and furthermore $Q\cong T_{\pp^2}(-1)$.

\begin{thm}\label{thm:ulrichsecantrational}
	 With the previous notation, the coherent sheaf \begin{equation*}
	 \mathcal{F}_n = \varepsilon_*p^*(S^{n-2}Q\otimes \mathcal{O}_{\pp^2}(-1)) \otimes \mathcal{O}_{\pp^n}(1)
	\end{equation*}
is a  Ulrich sheaf of rank $n-1$ on the secant $\Sigma$ of the rational normal curve in  $\pp^n$. Furthermore $D(\mathcal{F}_n) \cong \mathcal{F}_n$.
\end{thm}

\begin{proof}
To prove that the sheaf is Ulrich, we check the conditions of \Cref{lem:ulrichconditions} for $F=S^{n-2}Q\otimes \mathcal{O}_{\pp^2}(-1)$. We first observe that in this case
\begin{equation}
\label{eq:autoduality} 
F^{\vee}\otimes A_{2,\omega_C\otimes L} \cong S^{n-2}Q^{\vee}\otimes \mathcal{O}_{\pp^2}(1)\otimes \mathcal{O}_{\pp^2}(n-4) \cong S^{n-2}Q\otimes \mathcal{O}_{\pp^2}(-1) \cong F
\end{equation}
where the last isomorphism comes from the fact that $Q$ is a rank two bundle of determinant $\mathcal{O}_{\pp^2}(1)$, so that $Q^{\vee} \cong Q\otimes \mathcal{O}_{\pp^2}(-1)$. Now we check the conditions of \Cref{lem:ulrichconditions}: for the first condition, we show that $H^1(\ell,F_{|\ell})=0$ for every line in $\pp^2$, and this follows from the fact that $Q_{|\ell}\cong \mathcal{O}_{\ell}\oplus \mathcal{O}_{\ell}(1)$ for every line, so that $F_{|\ell} \cong \mathcal{O}_{\ell}(-1)\oplus \mathcal{O}_{\ell} \oplus \dots \oplus\mathcal{O}_{\ell}(n-3)$.

For the other conditions, we observe that in the notation of \Cref{lem:ulrichconditions}, we have $t=1$, so that, thanks to \Cref{eq:autoduality}, we are reduced to the two vanishings
\[ H^0(\pp^2,S^{n-2}Q(-1)) = H^1(\pp^2,S^{n-2}Q(-1)) = 0 \]
which hold because the bundle $S^{n-2}Q(-1)$ has no cohomology. This last statement follows from Borel--Weil--Bott in characteristic zero, but there is an easier proof valid in all characteristics, for which we thank Claudiu Raicu: taking symmetric powers in the exact sequence \Cref{eq:exseqQ}, we obtain another exact sequence
\[ 0 \to S^{n-3}(\mathcal{O}^{\oplus_3}_{\pp^2}) \otimes \mathcal{O}_{\pp^2}(-2) \to S^{n-2}(\mathcal{O}^{\oplus_3}_{\pp^2}) \otimes \mathcal{O}_{\pp^2}(-1) \to S^{n-2}Q \otimes \mathcal{O}_{\pp^2}(-1) \to 0 \]
and since the line bundles $\mathcal{O}_{\pp^2}(-1)$ and $\mathcal{O}_{\pp^2}(-2)$ have no cohomology, the statement follows.  This proves that $\mathcal{F}_n$ is Ulrich. The fact that $D(\mathcal{F}_{n}) \cong \mathcal{F}_n$ follows from 
\Cref{lem:duality} and \Cref{eq:autoduality}.
\end{proof}

\section{The arithmetic writhe}\label{sec:writhe}
\subsection{An algebraic orientation on \texorpdfstring{$B\setminus E$}{B-E}}
We use the notation from \Cref{sec:projbundles}. Namely, let $L$ be a $3$-very ample line bundle on the curve $C$ over the field $K$ and $\iota\colon C\hookrightarrow\pp(H^0(C,L)^\vee)=\pp^n$ be the corresponding embedding. We have the tautological projective bundle $p\colon B\to C_2$, i.e., $B=\pp(E_{2,L})$ where $E_{2,L}$ is the tautological bundle (see \Cref{sec:prelimcurves}) on the second symmetric power $C_2$ of $C$, and the projection $\varepsilon\colon  B\to\Sigma$ to the secant variety $\Sigma$ of $C$. The latter is an isomorphism outside the preimage $E$ of $C$.
Let $\alpha$ be a theta characteristic on $C$ such that $H^0(C,\alpha)=0$. We have seen in Theorem \ref{thm:ulrichsecant} that this defines a symmetric Ulrich sheaf of rank one on $\Sigma$, and consequently a relative orientation as in \Cref{thm:inducedorient}. We will now make this concrete and explain how $\alpha$ defines an algebraic orientation on $B\setminus E$, even when $\alpha$ has section. To that end recall from \Cref{cor:canonicalBn} that the canonical bundle on $B$ is given by
\begin{equation*}
    p^*(S_{2,\omega_C})\otimes\cO_B(-E)\cong p^*(S_{2,\alpha})\otimes p^*(S_{2,\alpha})\otimes\cO_B(-E).
\end{equation*}
This implies that $\omega_{B\setminus E}$ is isomorphic to the restriction of $p^*(S_{2,\alpha})\otimes p^*(S_{2,\alpha})$. Thus $p^*(S_{2,\alpha})|_{B\setminus E}$ is a theta characteristic on $B\setminus E$. In order to even define an algebraic orientation on $B\setminus E$, choose an isomorphism $\sigma\colon \alpha\otimes\alpha\to\omega_C$ for now. Let $\tau$ be a rational section of $\alpha$ and let $\omega=\sigma(\tau\otimes\tau)$. Let $\rho$ be a rational section of $p^*(S_{2,\alpha})$ with divisor $S_{\divv(\tau)}$. Then by \Cref{eq:concretediff} an isomorphism
\begin{equation*}
    \psi\colon   p^*(S_{2,\alpha})|_{B\setminus E} \otimes p^*(S_{2,\alpha})|_{B\setminus E} \to \omega_{B\setminus E}
\end{equation*}
can be defined by mapping $\rho\otimes\rho$ to the differential $3$-form $\Omega(\omega)$ from \Cref{thm:viroframe}. Note that by \Cref{cor:deponomega} another choice of $\sigma$, $\tau$ and $\rho$ would lead to an equivalent algebraic orientation. Therefore, we can call $\psi$ the \emph{algebraic orientation induced by $\alpha$}.

\subsection{An arithmetic count of secant lines}\label{sec:arithwr}
Let $x_0,\ldots,x_3$ be the homogeneous coordinates on $\pp^3$ and $x=\frac{x_1}{x_0}$, $y=\frac{x_2}{x_0}$ and $z=\frac{x_3}{x_0}$. We consider on $\pp^3$ the algebraic orientation defined by the rational differential $3$-form $dx\wedge dy\wedge dz$ (meaning that we choose an isomorphism for which this $3$-form is the image of a square).
Now let $\iota\colon C\hookrightarrow\pp^3$ be an embedded curve over $K$ such that $L=\iota^*\cO_{\pp^3}(1)$ is $3$-very ample. Geometrically, this means that $C$ is the isomorphic image via a linear projection of a curve $C\subset \mathbb{P}^r$ such that no four points on $C$ lie on a plane in $\mathbb{P}^r$.

Let $\alpha$ be a theta characteristic on $C$ and let $\Sigma \subset \mathbb{P}^r$ be the secant variety to $C$. Composing the map $p\colon B\to\Sigma$ with a suitable projection and restricting to $B\setminus E$, we obtain the finite surjective morphism $p'\colon B\setminus E\to\pp^3\setminus C$. The algebraic orientation on $B\setminus E$ induced by $\alpha$ together with our fixed algebraic orientation on $\pp^3$ define a relative orientation of $p'$ and since $\pp^3$ is $\A^1$-chain connected we have that $\deg^{\A^1}(p')$ is well-defined (see \cite[Theorem~9]{pauliwickelgren} applied to the proper map $B\to\pp^3$).

\begin{Def}
 We define the \emph{arithmetic writhe} of the $3$-very ample curve $C$ semi-oriented by $\alpha$ as $\w(C,\alpha):=\deg^{\A^1}(p')$.
\end{Def}

\begin{rem}\label{rem:localwrithesum}
 The arithmetic writhe is the result of an arithmetic count of secant lines $S$ to $C$ passing through a given point $q\in\pp^3\setminus C$. Indeed, such secant lines are in bijection to points $q_S$ in the preimage of $q$ under $p'$. We define the \emph{local writhe} $\w_{q,S}(C,\alpha)$ at such a secant to be the local $\A^1$-degree $\deg_{q_S}^{\A^1}(p')$. The sum of the local writhe over all secants that contain a point $q$ is independent of $q$ and equals $\w(C,\alpha)$. Note that $\w(C,\alpha)$ \emph{does} however depend on the embedded curve $C$ (and $\alpha$) as for example \Cref{prop:rational4} shows.
\end{rem}

We now describe how to compute the local writhe explicitly. By \Cref{thm:viroframe} $\w_{q,S}(C,\alpha)$ can be computed by evaluating $dx\wedge dy\wedge dz$ (or any other differential $3$-form that differs from this by the square of a rational function) at the $\omega$-Viro frame at $q_S$ where $\omega$ is any differential on $C$ whose divisor is of the form $2D$ where $D$ is a divisor whose line bundle is $\alpha$. More precisely, we choose tangent vectors $u\in\T_qS$, $v\in\T_aC$ and $w\in\T_bC$ as in \Cref{def:viroframe}. Namely, $v\in\T_aC$ and $w\in\T_bC$ are chosen in a way that $\omega$ evaluates to one in $v$ and $w$. After identifying the secant $S$ with $\pp^1$ sending $a$ to $(1:0)$, $b$ to $(0:1)$ and $q$ to $(1:1)$, the tangent vector $u$ is chosen in a way that $dt$ evaluates to one in $u$. Then write $u,v$ and $w$ as vectors $\tilde{u},\tilde{v}$ and $\tilde{w}$ with respect to the basis given by $dx$, $dy$ and $dz$. The local writhe then equals
\begin{equation*}
    \w_{q,S}(C,\alpha)=\Tr_{F/K}\langle\det\begin{pmatrix}
    |&|&| \\
    \tilde{v}&\tilde{u}&\tilde{w}\\
    |&|&|
    \end{pmatrix}\rangle
\end{equation*}
where $F$ is the field of definition of $S$.

\begin{rem}
 The above description of the arithmetic writhe as a sum of local writhe numbers shows that in the case $K=\R$ it agrees with the encomplexed writhe number introduced by Viro in \cite{virowrithe}.
 Realizing it as the degree of a morphism gives another proof that the encomplexed writhe number does not depend on the choice of the center of projection.
\end{rem}

\begin{ex}\label{ex:quarticwrith}
 Let $K$ a field with $\Char(K)\neq2$ and consider the rational curve $C\subset\pp^3$ of degree four defined as the image of
 \begin{equation*}
     \varphi\colon \pp^1\to\pp^3,\, (r:s)\mapsto(r^4:r^3s:rs^3:s^4).
 \end{equation*}
 In the following every (local) writhe number will be computed with respect to the unique theta characteristic on $C\cong\pp^1$ given by $\cO_{\pp^1}(-1)$.
 In order to compute the local writhe of secants to $C$, we choose the algebraic orientation $\sigma$ on $C$ given by $dt=t^2d\bar{t}$ where $t=\frac{s}{r}$ and $\bar{t}=-\frac{r}{s}$. 
 The genus-degree formula implies that the projection from a general point has three nodes. 
 We will now compute the local writhe of all three secants that contain the point $P=(1:0:0:1)$. 
 For the secant $L_1$ spanned by the points $a=\varphi(1:1)$ and $b=\varphi(1:-1)$ we can work on the affine chart $x_0\neq0$ and the coordinates $x,y,z$. In this chart the curve is parametrized by
 \begin{equation*}
     t\mapsto (t,t^3,t^4).
 \end{equation*}
Expressed in these coordinates our points $a$ and $b$ correspond to $\tilde{a}=(1,1,1)$ ($t=1$) and $\tilde{b}=(-1,-1,1)$ ($t=-1$). Tangent vectors at $\tilde{a}$ and $\tilde{b}$ in direction of $t$ are given by $\tilde{v}=(1,3,4)$ and $\tilde{w}=(1,3,-4)$. For computing the vector $\tilde{u}$ we have to consider the parametrization
\begin{equation*}
    \lambda\mapsto \left(\frac{1-\lambda}{1+\lambda},\frac{1-\lambda}{1+\lambda},1\right).
\end{equation*}
The tangent vector at $\lambda=1$ equals $\tilde{u}=(-\frac{1}{2},-\frac{1}{2},0)$.
We can thus compute the local writhe as
 \begin{equation*}
    \w_{p,L_1}(C)=\langle \det\begin{pmatrix}
     1&-\frac{1}{2}&1\\
     3&-\frac{1}{2}&3\\
     4&0&-4
     \end{pmatrix}\rangle=\langle -8\rangle=\langle-2\rangle.
 \end{equation*}
 Let $i\in\overline{K}$ a square root of $-1$.
 For the secant $L_2$ spanned by the points $a=\varphi(1:i)$ and $b=\varphi(1:-i)$ we can work on the same affine chart. Note that although $a$ and $b$ might not be $K$-rational points, the line $L_2$ is defined over $K$. We have $\tilde{a}=(i,-i,1)$ ($t=i$) and $\tilde{v}=(1,-3,-4i)$ as well as $\tilde{b}=(-i,i,1)$ ($t=-i$) and $\tilde{w}=(1,-3,4i)$. In this chart the line $L_2$ is parametrized by 
 \begin{equation*}
    \lambda\mapsto \left(i\cdot\frac{1-\lambda}{1+\lambda},-i\cdot\frac{1-\lambda}{1+\lambda},1\right). 
 \end{equation*}
 Thus we compute $\tilde{u}=(-\frac{i}{2},\frac{i}{2},0)$.
 Now we can compute the local writhe as
 \begin{equation*}
    \w_{p,L_2}(C)= \langle\det\begin{pmatrix}
     1&-\frac{i}{2}&1\\
     -3&\frac{i}{2}&-3\\
     -4i&0&4i
     \end{pmatrix}\rangle=\langle 8\rangle=\langle2\rangle.
 \end{equation*}
 For the third secant $L_3$ spanned by $\varphi(1:0)$ and $\varphi(0:1)$ we work on the affine chart $x_0+x_3\neq0$ and the coordinates $\frac{x_1}{x_0+x_3},\frac{x_2}{x_0+x_3},\frac{x_3}{x_0+x_3}$. A short computation verifies that these are compatible with our chosen algebraic orientation. In this chart, the curve is parametrized by
 \begin{equation*}
     t\mapsto\left(\frac{t}{1+t^4},\frac{t^3}{1+t^4},\frac{t^4}{1+t^4}\right)
     \quad \text{resp.} \quad
\end{equation*}    
     \begin{equation*}
     \bar{t}\mapsto\left(\frac{-\bar{t}^3}{1+\bar{t}^4},\frac{-\bar{t}}{1+\bar{t}^4},\frac{1}{1+\bar{t}^4}\right).
 \end{equation*}
 We have $\tilde{a}=(0,0,0)$ ($t=0$) and $\tilde{b}=(0,0,1)$ ($\bar{t}=0$). Tangent vectors at $\tilde{a}$ and $\tilde{b}$ in direction of $t$ resp. $\bar{t}$ are given by $\tilde{v}=(1,0,0)$ and $\tilde{w}=(0,-1,0)$.
 Moreover, in this chart the line $L_3$ is parametrized by 
 \begin{equation*}
    \lambda\mapsto \left(0,0,\frac{\lambda}{1+\lambda}\right)
 \end{equation*}
 which gives $\tilde{u}=(0,0,\frac{1}{4})$.
 We can thus compute the local writhe as
  \begin{equation*}
     \w_{p,L_3}(C)=\langle\det\begin{pmatrix}
     1&0&0\\
     0&0&-1\\
     0&\frac{1}{4}&0
     \end{pmatrix}\rangle=\langle\frac{1}{4}\rangle=\langle 1\rangle.
 \end{equation*}
 Summing up the local writhe numbers we find that the arithmetic writhe of $C$ is 
 \begin{equation*}
 \w(C)=\langle -2\rangle+\langle 2\rangle+\langle 1\rangle=2\cdot\langle 1\rangle+\langle -1\rangle.
 \end{equation*}
\end{ex}
\begin{figure}
    \centering
    \includegraphics[width=5cm]{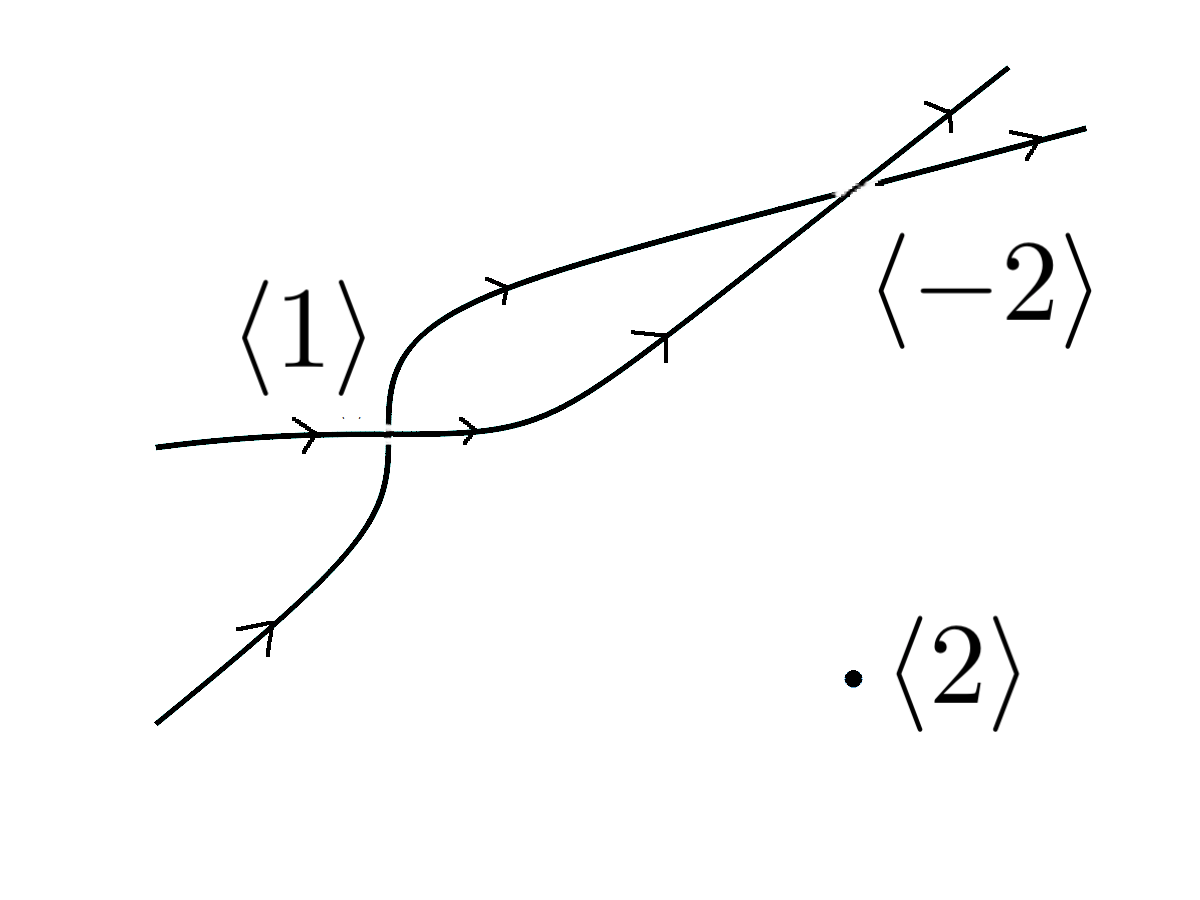}
    \caption{The quartic curve considered in \Cref{ex:quarticwrith} projected from the point $P$. The nodes of this planar curve correspond to secant lines containing $P$ and we denoted their local writhe in the picture.}
    \label{fig:quarticwrith}
\end{figure}

If the theta characteristic $\alpha$ does not have global sections, then we can compute the arithmetic writhe of $C$ directly from the Ulrich bundle constructed in \Cref{thm:ulrichsecant}.

\begin{thm}\label{thm:writheulrich}
 Let $L$ be a $3$-very ample line bundle on the curve $C$ and $\alpha$ a theta characteristic of $C$ with
 $h^0(C,\alpha)=0$. There exists a symmetric matrix $\Lambda$ whose entries are linear forms on $\wedge^{4}H^0(C,L)$ such that for all $s_0\ldots,s_3 \in H^0(C,L)$ it holds that:
 \begin{enumerate}
     \item $\Lambda(s_0\wedge\cdots\wedge s_3)$ is non-singular if and only if the linear system spanned by the $s_0,\ldots,s_3$ is very ample.
     \item If the linear system spanned by the $s_0,\ldots,s_3$ is very ample, then the writhe $\w(\psi(C),\alpha)$ is the class of $\Lambda(s_0\wedge\cdots\wedge s_3)$ where $\psi\colon C\to\pp^3$ is the embedding defined by $s_0,\ldots,s_3$.
 \end{enumerate}
\end{thm}

\begin{proof}
 Let $\Sigma$ the secant variety of $C$ embedded to $\pp(H^0(C,L)^\vee)$. Then the linear system spanned by $s_0,\ldots,s_3\in H^0(C,L)$ is very ample if and only if $s_0,\ldots,s_3$ do not have a common zero on $\Sigma$. If so, then $\w(\psi(C),\alpha)$ is the $\A^1$-degree of the linear projection $\Sigma\to\pp^3$ given by $s_0,\ldots,s_3$ restricted to the preimage of $\pp^3\setminus C$ relatively oriented by the rank one symmetric Ulrich sheaf from \Cref{thm:ulrichsecant}. Hence the claim follows from \Cref{thm:a1chowmatrix}.
\end{proof}

We end this section with some thoughts on which classes $w\in\GW(K)$ can be realized as the writhe of a spatial curve.

\begin{prop}\label{prop:rational4}
 For every $a,b,c\in K^\times$ there is a rational curve $X\subset\pp^3$ of degree four with $\w(X)=\langle a\rangle+\langle b\rangle+\langle c\rangle$.
\end{prop}

\begin{proof}
 Let $C=\pp^1$, $L=\cO_{\pp^1}(4)$ and $\alpha=\cO_{\pp^1}(-1)$. The matrix $\Lambda$ from \Cref{thm:writheulrich} is then given by the Hankel matrix
 \begin{equation}\label{eq:hankel}
     \Lambda=\begin{pmatrix}
       x_0&x_1&x_2\\
       x_1&x_2&x_3\\
       x_2&x_3&x_4
     \end{pmatrix}
 \end{equation}
 where $x_i$ denotes the linear form  
 \begin{equation*}
     f\mapsto f\wedge (r^{4-i}s^i)\in\wedge^5K[r,s]_4=K
 \end{equation*}
 on $\wedge^4K[r,s]_4$.  On the other hand, for arbitrary $a,b,c\in K^\times$ we have
 \begin{equation*}
     \begin{pmatrix}
     a+\frac{b^2}{c}& 0& b\\
     0& b& 0\\
     b& 0& c  
     \end{pmatrix}
     \cong \langle a\rangle+\langle b\rangle+\langle c\rangle.
 \end{equation*}
 This proves the claim.
\end{proof}

\begin{ex}
    We have for instance
 \begin{equation*}
     \Lambda(r^4\wedge r^3s\wedge rs^3\wedge s^4)=\begin{pmatrix}
       0&0&1\\
       0&1&0\\
       1&0&0
     \end{pmatrix}\cong 2\cdot\langle 1\rangle+\langle -1\rangle
 \end{equation*}
 which confirms our calculations from \Cref{ex:quarticwrith}.
\end{ex}

 \begin{qu}
   Let $d\in\N$, $N=\frac{1}{2}(d-1)(d-2)$ and $a_1,\ldots,a_N\in K^\times$. Is there a rational curve $X\subset\pp^3$ of degree $d$ with $\w(X)=\sum_{i=1}^N\langle a_i\rangle$? More generally, for any $N\in\N$ and $a_1,\ldots,a_N\in K^\times$, does there exist a curve $X\subset\pp^3$ and a theta characteristic $\alpha$ on $X$ with $\w(X,\alpha)=\sum_{i=1}^N\langle a_i\rangle$?
 \end{qu}

\section{Algebraic isotopies}\label{sec:isotopies}

Let $C$ be a curve over a field $K$ of characteristic $\Char(K)\neq2$. In analogy to classical knot and link theory, we want to study embeddings of $C$ into the three-dimensional space $\pp^3$ up to algebraic isotopies. 
We first make some general definitions and observations.

\begin{Def}
	Let $X$ be a projective variety over $K$. An \emph{algebraic isotopy} of two closed embeddings $\psi_i\colon X\to \pp^r$, $i=0,1$, is a morphism
	\begin{equation*}
		I\colon \A^1\times X \to  \pp^r,\qquad (t,x) \mapsto I_t(x) 
	\end{equation*}
	such that for all $t\in\A^1$ the induced map $I_{t}\colon X_{\kappa(t)}\to Y_{\kappa(t)}$ is an embedding and $I_i=\psi_i$ for $i=0,1$.
	Two embeddings $\psi,\psi'$ are \emph{algebraically isotopic} if they are connected by a chain of algebraic isotopies. We denote by $[X,\pp^r]^{\iso}$ the quotient of the set of embeddings $\operatorname{Emb}(X,\pp^r)(K)$ by the equivalence relation generated by algebraic isotopies. We write $[\psi]$ for the class of $\psi$ in $[X,\pp^r]^{\iso}$ and call it the \emph{isotopy class} of $\psi$.
\end{Def}

A first observation is that two isotopic embeddings have isomorphic underlying line bundles:

\begin{lem}
	Let $\psi,\psi'\colon X\hookrightarrow \pp^r$ be two algebraically isotopic embeddings. Then $\psi^*\mathcal{O}_{\pp^r}(1)\cong \psi'^*\mathcal{O}_{\pp^r}(1)$.
\end{lem} 
\begin{proof}
	Any algebraic isotopy $I\colon \A^1\times X \to \pp^r$ induces a morphism $I\colon \A^1 \to \operatorname{Pic}(X),t\mapsto I_t^*\mathcal{O}_{\pp^r}(1)$, which is constant, since any morphism from a rational curve to an abelian variety is constant. Alternatively, the map
    \[ \pr_X^*\colon \operatorname{Pic}(X) \longrightarrow \operatorname{Pic}(\A^1\times X),\quad L \mapsto \pr_X^*L \]
    is an isomorphism \cite[Proposition II.6.6]{Hart77}, so that the isomorphism class of the line bundle $I^*\mathcal{O}_{\pp^r}(1)$ is constant on the fibers $\{t\}\times X$.
\end{proof}

\begin{Def}
	Let $X$ be a projective variety over $K$ and $L$ a line bundle on $X$. We denote the set of embeddings $\psi \colon X\hookrightarrow \pp^r$ with $\psi^*\mathcal{O}_{\pp^n}(1)\cong L$ by $\operatorname{Emb}_L(X,\pp^r)(K)$, and its quotient by chains of algebraic isotopies by $[X,
    \pp^r]^{\iso}_L$.
\end{Def}

We now recall how to parametrize the set $\operatorname{Emb}_L(X,\pp^r)(K)$.
First of all, if this is nonempty, then $L$ is very ample. In this case, consider the embedding given by the complete linear system $V=H^0(X,L)$:
\[ \varphi_L \colon X \hookrightarrow \pp^n = \pp(V^{\vee}) \]
and the secant variety $\Sigma = \Sigma(X,L)$ of $X$ in $\pp(V^{\vee})$. Any embedding $\psi$ from $\operatorname{Emb}_L(X,\pp^r)$ is obtained by composing $\varphi_L$ with a linear projection
\[ [s_0,\dots,s_r]\colon \varphi_L(X) \longrightarrow \pp^r, \qquad s_i \in V\]
and, in order for $\psi = [s_0,\dots,s_r]\circ \varphi_L$ to be an embedding, the condition is that the linear space $\Pi = \{ s_0 = \dots = s_r = 0 \}$ does not intersect the secant variety $\Sigma$. Consider now the set
\begin{equation}\label{eq:ULr} 
\mathcal{U}_{L,r} = \{(s_0,\dots,s_r) \in V^{\oplus (r+1)} \,|\, \{ s_0 = \dots = s_r = 0\} \cap \Sigma = \emptyset \} 
\end{equation}
This is open: its complement in $V^{\oplus (r+1)}$ is the image of the incidence correspondence 
\[ \mathcal{I} = \{((s_0,\dots,s_r),p) \in V^{\oplus (r+1)}\times \Sigma \,|\, s_0(p)=\dots=s_r(p)=0  \} \]
along the proper map $\pr_{V^{\oplus (r+1)}}\colon V^{\oplus(r+1)}\times \Sigma \to V^{\oplus(r+1)}$. The previous discussion shows that $\operatorname{Emb}_L(X,\pp^r)(K)$ is the set of $K$-rational points of a  (possibly empty) open subset of a projective space
\begin{equation}\label{eq:embL} 
\operatorname{Emb}_L(X,\pp^r) = \mathcal{U}_{L,n}/K^{\times} = \{[s_0,\dots,s_r] \in \pp(V^{\oplus (r+1)}) \,|\, (s_0, \dots,s_r) \in \mathcal{U}_{L,r} \} . 
\end{equation}
This comes equipped with a universal embedding
\[ \operatorname{Emb}_L(X,\pp^r) \times X \longrightarrow \pp^r, \qquad ([s_0,\dots,s_r],x) \mapsto [s_0(x),\dots,s_r(x)]\]
so that an algebraic isotopy between $\psi,\psi'\in \operatorname{Emb}_L(X,\pp^r)(K)$ is simply a map $\A^1 \to \operatorname{Emb}_L(X,\pp^r)$ whose image contains these two points. This leads naturally to the notion of naive connected components in $\A^1$-geometry as  e.g. in \cite{cazanave}.

\begin{Def}\label{def:naiveconncomp}
	Let $X$ be a scheme over the field $K$. The set $\pi_0^{\naiv}X(K)$ of \emph{naive connected components} of $X(K)$ is the quotient of $X(K)$ by the finest equivalence relation under which each two points $x_0,x_1\in X(K)$ for which there exists a morphism $f\colon \A^1\to X$ with $f(i)=x_i$, $i=0,1$, are equivalent.
\end{Def}

\begin{rem}\label{rem:pi0emb}
The previous discussion shows that the set of isotopy classes of embeddings with line bundle $L$ can be considered as the set of naive connected components
\[ [X,\pp^r]^{\iso}_{L}  = \pi_0^{\naiv} \operatorname{Emb}_L(X,\pp^r)(K) \]
Note that we could also replace $\pp^r$ by any other projective variety $Y$ by using Grothendieck's Hom-scheme, although we will not consider this here.
\end{rem}

\begin{rem}
Recall that a $K$-scheme $X$ is $\A^1$-connected if and only if for any extension $K'/K$, any two $K'$-rational points in $X$ can be connected by a chain of maps from $\A^1_{K'}$. This means precisely that $\pi_0^{\naiv} \operatorname{Emb}_L(X_{K'},\pp^r)(K')$ has at most one element for any extension $K'/K$.
\end{rem}

\begin{ex}\label{ex:slconnected}
The group scheme $\SL_{r+1}$ is $\A^1$-connected: indeed, for any field $K$, any matrix in $\SL_{r+1}(K)$ can be written as a product of upper or lower triangular elementary matrices with all ones on the diagonal. For such a matrix $A$, the map
\[ \A^1 \longrightarrow \SL_{r+1}, \qquad t \mapsto I_{r+1}+t(A-I_{r+1}) \]
connects $A$ at $t=1$ with the identity $I_{r+1}$ at $t=0$.
\end{ex}

\begin{cor}\label{cor:slconnected}
Let $\psi \colon X\hookrightarrow \pp^r$ be an embedding and let $A\in \SL_{r+1}(K)$. The two embeddings $\psi$ and $A\circ \psi$ are algebraically isotopic.
\end{cor}
\begin{proof}
Since $\SL_{r+1}$ is $\A^1$-connected, we can find a chain of algebraic isotopies connecting $A$ with the identity $I_{r+1}$. 
\end{proof}

\begin{lem}\label{lem:subspaceisotopic}
	Let $\psi \colon X \hookrightarrow \pp^r$ be an embedding and let $A\in K^{(r+1)\times (r+1)}$ of non-maximal rank, such that the composition of $\psi$ with the rational map $A\colon \pp^r \dashrightarrow \pp^r$  is still an embedding. Then $\psi$ and $A\circ \psi$ are algebraically isotopic.
\end{lem}
\begin{proof}
   Let $\operatorname{rk}(A) = m+1$ with $m<r$. Then there is an $M\in \GL_{r+1}(K)$ such that $MA$ is in reduced row-echelon form:
   \[  M\cdot A = \begin{pmatrix} A' \\ 0 \end{pmatrix}\]
   where $A'$ is a $(m+1)\times (r+1)$ matrix. If we let $M'\in \GL_{m-r}(K)$ be any matrix with determinant $\det(M')=\det(M)^{-1}$, we see that
   \[  N = \begin{pmatrix} I_{m+1} & 0 \\ 0 & M' \end{pmatrix}, \text{ satisfies } NM\in \SL_{r+1}(K) \text{ and }  NM\cdot A = \begin{pmatrix} A' \\ 0 \end{pmatrix}\]
    Thus, thanks to Corollary \ref{cor:slconnected} we can assume that $A$ has reduced row-echelon form. 
     For simplicity, let's assume that the pivots of $A$ are in the first $m+1$ columns, so that \[ A=\begin{pmatrix} I_{m+1} & B \\ 0 & 0 \end{pmatrix} \] for another $(m+1)\times (n-m)$ matrix $B$, with the general case being similar. There is an algebraic isotopy
	\[ \A^1 \times X \to \pp^n, \qquad (t,x) \mapsto \begin{pmatrix} I_{m+1} & B \\ 0 & tI_{n-m} \end{pmatrix}\cdot \psi(x) \]
	between $A\circ\psi$ at $t=0$ and another embedding $\psi'$ at $t=1$. To conclude we need to show that $\psi'$ is algebraically isotopic to $\psi$. It is enough to observe that 
	\[ \begin{pmatrix} I_{m+1} & -B \\ 0 & I_{n-m} \end{pmatrix} \cdot \begin{pmatrix} I_{m+1} & B \\ 0 & I_{n-m} \end{pmatrix}\cdot \psi(x) = \begin{pmatrix} I_{m+1} & 0 \\ 0 & I_{n-m} \end{pmatrix}\cdot \psi(x) = \psi(x) \]
    so that we conclude by invoking Corollary \ref{cor:slconnected} again.
\end{proof}

The following corollary implies the first part of Theorem D.

\begin{cor}\label{cor:isotrivial}
	Let $L$ be a line bundle on $X$ with $h^0(X,L)\leq r$. Any two embeddings $\psi,\psi' \colon X\hookrightarrow \pp^n$ in $\operatorname{Emb}_L(X,\pp^n)$ are algebraically isotopic. Equivalently, $[X,\pp^r]^{\iso}_{L}$ has at most one element.
\end{cor}
\begin{proof}
	If $L$ is not very ample, there is nothing to prove. If $L$ is very ample, consider the embedding $\varphi_V\colon X\hookrightarrow \pp(V^{\vee})$. Since $\dim \pp(V^{\vee}) <r$ we can compose this with a linear embedding $j\colon\pp(V^{\vee}) \hookrightarrow \pp^r$ to get another embedding $\psi''=j\circ \varphi_L$. By construction, there is an $A\in K^{(r+1)\times (r+1)}$ not of maximal rank, such that $\psi = A\circ 
    \psi''$. Then Lemma \ref{lem:subspaceisotopic} shows that the resulting embedding $\psi''$ and $\psi$ are algebraically isotopic. The same reasoning holds for $\psi'$. 
\end{proof}

\begin{cor}\label{cor:embsecantiso}
	Consider an embedding $\psi\colon X\hookrightarrow \pp^r$ in $\operatorname{Emb}_L(X,\pp^r)$ and let $\langle \psi(X) \rangle$ be the linear span. Consider also the secant variety $\Sigma$ of $\varphi_L(X)$ in $\pp(V^{\vee})$.
	\begin{enumerate}
		\item It holds that $\dim \langle \psi(X) \rangle \geq \dim \Sigma$.
		\item If $K$ is infinite, then $\psi$ is algebraically isotopic to an embedding $\psi'\colon X\hookrightarrow \pp^r$ such that $\dim \langle \psi'(X) \rangle = \dim \Sigma$.
	\end{enumerate}
\end{cor} 
\begin{proof}
	By the description in \Cref{eq:embL} we know that $\psi = [s_0,\dots,s_r]$, with $\Pi = \{s_0 = \dots = s_r = 0 \} \subset \pp(V^{\vee})$ disjoint from $\Sigma$. Furthermore, the dimension of the linear span $\langle \psi(X) \rangle$ is 
    \[ \dim \langle \psi(X) \rangle =  \dim \langle s_0,\dots,s_r \rangle-1 = \operatorname{codim}((\Pi,\pp(V^{\vee}))-1 \]
    where $\langle s_0,\dots,s_n \rangle \subset V$ is the vector subspace generated by the $s_i$.
    
	\begin{enumerate}
		\item Since $\Pi\cap \Sigma = \emptyset$, it must be that $\codim (\Pi,\pp(V^{\vee})) \geq \dim \Sigma+1$. Then
		$\dim \langle \psi(X) \geq \dim \Sigma$.
		\item Assume that $\dim \langle \psi(X)\rangle > \dim \Sigma$. Since the image of the map 
        \begin{equation*}
         [s_0,\dots,s_r]\colon \Sigma \to \pp^r   
        \end{equation*}
        is the secant variety $\Sigma'$ to $\psi(C)$, we see that $\dim \Sigma'<r$. Since $K$ is infinite, we can find a $K$-point in $\langle \psi(X)\rangle \setminus \Sigma'$ and projecting from this point onto a hyperplane $H\subset \langle \psi(X)\rangle $ we get an embedding $\psi'\colon X\hookrightarrow H$, which is algebraically isotopic to $\psi$, thanks to \Cref{lem:subspaceisotopic}. By repeating this procedure, we conclude. \qedhere
	\end{enumerate} 
\end{proof}

In the following, we fix a very ample line bundle $L$ on $X$, the space of global sections $V=H^0(X,L)$ and we consider the secant variety $\Sigma(X,L)$ of $X$ inside $\pp(V^{\vee})$. We set 
\[ b = \dim \Sigma(X,L)\]

\begin{rem} One always has 
\[ b\leq \min\{ \dim \pp(V^{\vee}),2\cdot \dim X+1\}\]
and when the equality holds, then the secant variety has expected dimension. This happens for example if $X$ is a curve \cite{lange} or if $X$ is arbitrary and the line bundle $L$ is $3$-very ample. 
\end{rem} 

Now we want to strengthen \Cref{cor:isotrivial} and show that, if $K$ is infinite, the isotopy classes of embeddings $X\hookrightarrow \pp^r$ with $r>b$ are trivial. {This is analogous to the fact that any knot can be unraveled in dimension higher than three}.

\begin{prop}\label{prop:isotrivial} 
Assume that $K$ is infinite and let $r>b$. Then any two embeddings $\psi,\psi'\colon X\hookrightarrow \pp^r$ in $\operatorname{Emb}_L(X,\pp^r)$  are algebraically isotopic. Hence, if $r\ne b$, the set $[X,\pp^r]^{\iso}_{L}$ has at most one element. In particular, if $X$ is a curve, then $[X,\pp^r]^{\iso}_{L}$ has at most one element whenever $r\geq4$.
\end{prop}
\begin{proof}
Since $K$ is infinite, we can use \Cref{cor:embsecantiso} and assume that both $\psi,\psi'$ have images contained in a linear subspace of dimension $b$. Because the group $\SL_{r+1}(K)$ acts transitively on these subspaces, we can use \Cref{cor:slconnected} and assume that this subspace is the one where all coordinates are zero apart from the first $b+1$. This means
\[ \psi = [s_0,\dots,s_b,0,\dots,0], \qquad \psi' = [s'_0,\dots,s'_b,0,\dots,0] \]
for some $s_i,s'_i\in V$. Observe that in the above expressions there is a positive number of zeroes, since we are assuming $b<r$. We claim that there is an $s_0''\in V$ such that 
\[ [s_0'',s_1,\dots,s_b,0,\dots,0], [s_0'',s_1',\dots,s_b',\dots,0]\colon X\hookrightarrow \pp^r \] 
are both embeddings. Assume for a moment that this claim is true: then \Cref{lem:subspaceisotopic} shows that both $\psi$ and $[s_0'',s_1,\dots,s_b,0,\dots,0]$ are algebraically isotopic to the embedding given by 
\[ [s_0,s_1,\dots,s_b,s_0'',0,\dots,0]\colon X\hookrightarrow \pp^r \] 
Notice that this is the step where we use that $b<r$. An analogous statement holds for $\psi'$, so that we can assume that $\psi$ and $\psi'$ have the form
\[ \psi = [s''_0,s_1,\dots,s_b,0,\dots,0], \qquad \psi' = [s''_0,s'_1,\dots,s'_b,0,\dots,0] \]
and by repeating this reasoning with other sections $s''_1,\dots,s''_b$, we conclude. We are left with proving the above claim. Consider the set 
\begin{align*}  
U &= \{s_0''\in V \,|\, [s_0'',s_1,\dots,s_b,0,\dots,0] \text{ is an embedding } \} \\
&= \{s_0''\in V \,|\, \{ s_0''=s_1=\dots=s_b=0\} \cap \Sigma = \emptyset \}.
\end{align*}
One can show as for the set $\mathcal{U}_{L,n}$ of \eqref{eq:ULr} that $U$ is open in $V$, and since $s_0\in U$ and $V$ is irreducible, $U$ is a dense open subset. The same reasoning shows that
\[ U' = \{s_0''\in V \,|\, [s_0'',s'_1,\dots,s'_b,0,\dots,0] \text{ is an embedding } \} \]
is a dense open subset, and since $K$ is infinite, there must be a $K$-rational point in $U\cap U''$, which is the section $s_0''$ that we are looking for. 

The last statement follows from what we have just proved for $r>b$ and from \Cref{cor:embsecantiso} for $r<b$.
\end{proof}

In light of \Cref{prop:isotrivial}, we will only consider the classes of embeddings 
\[ X \hookrightarrow \pp^b \qquad \text{ where } b=\dim \Sigma. \]
By \Cref{cor:embsecantiso}, any such embedding is non-degenerate, meaning that if $\psi = [s_0,\dots,s_b]$, the sections $s_i$ must be linearly independent. Now, consider the morphism of affine spaces
\[ w\colon \A(V^{\oplus (b+1)}) \to \A(\wedge^{b+1} V), \qquad (s_0,\dots,s_b) \mapsto s_0\wedge \dots \wedge s_b \]
and let $\mathscr{C}_{\Sigma}$ be the Chow form of the secant variety: this means that $\Pi = \{s_0=\dots=s_b=0\}$ intersects $\Sigma$ if and only if $\mathscr{C}_{\Sigma}(s_0\wedge \dots \wedge s_b) = 0$. Define then
$\mathcal{W}_{L,b}\subset\A(\wedge^b V)$ to be the intersection of the principal open affine set $\{\Ch_{\Sigma}\ne 0\}$ with the locus of decomposable multi-vectors: we have that $\mathcal{U}_{L,b} = w^{-1}(\mathcal{W}_{L,b})$, where $\mathcal{U}_{L,b}$ is the open subset in \eqref{eq:ULr}. We have seen in \eqref{eq:embL} and \Cref{rem:pi0emb}, that $\operatorname{Emb}_L(X,\pp^b) = \mathcal{U}_{L,b}/K^{\times}$ and that $[X,\pp^b]^{\iso}_L = \pi_0^{\naiv}\operatorname{Emb}_L(X,\pp^b)(K)$. Thus, we have a surjective map

\begin{equation}\label{eq:nvipara1}
\pi_0^{\naiv}\mathcal{U}(K)\to[X,\pp^b]^{\iso}_L, \qquad (\sigma_0,\dots,\sigma_b) \mapsto \psi = [\sigma_0,\dots,\sigma_b]
\end{equation}
Our next goal is to show that it factors through the map $\pi_0^{\naiv}\mathcal{U}_{L,b}(K)\to\pi_0^{\naiv}\mathcal{W}_{L,b}(K)$ induced by $w$.

\begin{lem}\label{lem:lift}
	Let $0\leq d< h^0(X,L)$ and $f\colon\A^1\to\A(\wedge^{d+1}V)$ be a morphism whose image is contained in the locus of decomposable multi-vectors. There exists a morphism $g\colon\A^1\to\A(V^{\oplus d+1})$ such that $f=w\circ g$:
	\begin{equation*}
		\begin{tikzcd}
		& \A(V^{\oplus d+1})\arrow{d}{w} \\
		\A^1\arrow{r}{f}\ar[ur, dashed, "\exists\, g"] &   \A(\wedge^{d+1}V)
		\end{tikzcd}
	\end{equation*}
\end{lem}

\begin{proof}
	Set $\ell=h^0(X,L)$. Choosing a basis, we can identify $V=K^{\ell}$. We let $M$ be the subsets of $\{1,\ldots,r\}$ with $d+1$ elements. For $J=\{j_0,\ldots,j_d\}$ with $1\leq j_0<\cdots<j_d\leq \ell$ we denote $e_J=e_{j_0}\wedge\cdots\wedge e_{j_d}$ where $e_i\in K^{\ell}$ is the $i$-th unit vector. The morphism $f$ is given by 
	\begin{equation*}
		f(t)=\sum_{J\in M} P_J(t)\cdot e_J 
	\end{equation*}
	for some polynomials $P_J(t)\in K[t]$. Let $G(t)\in K[t]$ the greatest common divisor of the $P_J(t)$. Since the image is contained in the locus of decomposable multi-vectors, we can write
	\begin{equation*}
		f(t)=v_0(t)\wedge\cdots\wedge v_d(t)
	\end{equation*}
	for some $v_j(t)\in K(t)^{\ell}$. We need to show that one can even choose $v_j\in K[t]^{\ell}$. Let $A(t)$ be the $\ell\times (d+1)$ matrix whose columns are the $v_j(t)$. Let $Q(t)\in K[t]$ be the smallest common denominator of the entries of $A$ and let $A'(t)=Q(t)\cdot A(t)$. Because $K[t]$ is a principal ideal domain, there exist matrices $S(t)\in\SL_{\ell}(K[t])$ and $T(t)\in\SL_{d+1}(K[t])$ such that
	\begin{equation*}
		S(t)\cdot A'(t)\cdot T(t)=\begin{pmatrix} D(t) \\ 0            \end{pmatrix}
	\end{equation*}
	where $D(t)$ is a $(d+1)\times (d+1)$ diagonal matrix with diagonal entries $p_0,\ldots,p_d\in K[t]$. Let $w_0(t),\ldots,w_d(t)$ be the first $d+1$ columns of the matrix $S(t)^{-1}$. Then 
	\begin{equation*}
		Q(t)\cdot f(t) = p_0(t)\cdots p_d(t)\cdot (w_0(t)\wedge\cdots\wedge w_d(t)).
	\end{equation*}
	Since $S(t)^{-1}$ is invertible over $K[t]$, the greatest common divisor of the Pl\"ucker coordinates of the right-hand side is $p_0(t)\cdots p_d(t)$. By construction, the greatest common divisor of the Pl\"ucker coordinates of the left-hand side is $G(t)\cdot Q(t)$. Thus, there is some $\lambda\in K^\times$ such that
	\begin{equation*}
		f(t) = \lambda\cdot G(t)\cdot (w_0(t)\wedge\cdots\wedge w_d(t)).
	\end{equation*}
	This proves the claim.
\end{proof}

\begin{cor}\label{cor:unwnbij}
	The map $\pi_0^{\naiv}\mathcal{U}_{L,b}(K)\to\pi_0^{\naiv}\mathcal{W}_{L,b}(K)$ induced by $w$ is a bijection.
\end{cor}

\begin{proof}
	It is surjective because $w$ is surjective on the level of $K$-points. For proving injectivity let $x,y\in \mathcal{U}_{L,b}(K)$ such that $w(x)$ and $w(y)$ belong to the same equivalence class in $\pi_0^{\naiv}\mathcal{W}_{L,b}(K)$. By \Cref{lem:lift} there exists $y'$ in the equivalence class of $y$ such that $w(y')=w(x)$. The $K$-points of any fiber of $w$ over a nonzero point form an $\SL_{b+1}(K)$-orbit. Therefore, by \Cref{ex:slconnected}, the two points $x$ and $y'$ and thus $x$ and $y$ belong to the same equivalence class in $\pi_0^{\naiv}\mathcal{U}_{L,b}(K)$.
\end{proof}

\begin{cor}\label{cor:scalingpower4}
	Let $(s_0,\ldots,s_n),(s_0',\ldots,s_n')\in \mathcal{U}_{b,L}(K)$. Then the embeddings 
	\begin{equation*}
		\psi,\psi'\colon X\hookrightarrow \pp^b   
	\end{equation*}
	defined by the $s_i$ and $s_i'$, respectively, are algebraically isotopic if and only if there exists $\lambda\in K^\times$ such that $s_0\wedge\cdots\wedge s_b$ and $\lambda^{b+1}\cdot(s_0'\wedge\cdots\wedge s_{b+1}')$ are in the same naive connected component of $\mathcal{W}_{b,L}(K)$. In other words, there is a natural bijection
    \[ \pi_0\mathcal{W}_{b,L}(K)/(K^{\times b+1}) \overset{\sim}{\longrightarrow} [X,\pp^b]^{\iso}_L. \]
\end{cor}

\begin{proof}
	Since $\operatorname{Emb}_{L}(X,\pp^r)(K) = \mathcal{U}_{L,b}(K)/K^{\times}$, we see that $\psi$ and $\psi'$ are algebraically isotopic if and only if and only if there exists $\lambda\in K^\times$ such that $(s_0,\ldots,s_b)$ and $(\lambda s_0',\ldots,\lambda s_b')$ are in the same naive connected component of $\mathcal{U}_{b,L}(K)$. By \Cref{cor:unwnbij} this is the case if and only if $s_0\wedge\cdots\wedge s_b$ and $\lambda^{b+1}\cdot( s_0\wedge\cdots\wedge s_{b+1})$ are in the same naive connected component of $\mathcal{W}_{b,L}(K)$.
\end{proof}

The following corollary implies the second part of Theorem D.

\begin{cor}
	Assume that $h^0(X,L)=b+1$. There is a bijection 
	\[ K^\times/K^{\times b+1}\to [X,\pp^{b}]^{\iso}_L.\]
\end{cor}

\begin{proof}
	Since $\dim(V)=b+1$, we have that $\mathcal{W}\cong\A^1\setminus\{0\}$. This shows that every point of $\mathcal{W}_{L,b}(K)$ is its own naive connected component and $\mathcal{W}_{L,b}(K)$ can be identified with $K^\times$. The statement follows then from  \Cref{cor:scalingpower4}.
\end{proof}

\section{Isotopic embeddings of curves in \texorpdfstring{$3$}{3}-space}\label{sec:isospacecurves}

In the following $C$ will be a smooth and connected projective curve over $K$. It follows from example from \Cref{cor:embsecantiso} that the secant variety of $C$ has the expected dimension three if and only if $C$ is not planar.
Thus, in the following we will consider embeddings
\[ \psi\colon C\hookrightarrow \pp^3 \]
such that $L = \psi^*\mathcal{O}_{\pp^3}(1)$ has $h^0(C,L)\geq 4$ and the isotopy classes in $[C,\pp^3]^{\iso}_L$, so that the results of the previous section apply, in particular \Cref{cor:scalingpower4}.

In the case of curves, we will prove that the arithmetic writhe is invariant under algebraic isotopies. We will make use of the following version of Harder's theorem.

\begin{thm}\label{thm:cazaharder}
	Let $n\in\N$ and $\mathscr{S}_n$ the scheme of non-singular symmetric $n\times n$ matrices over $K$. We consider the fiber product $\GW(K)\times_{K^\times/K^{\times2}} K^\times$ with respect to the canonical map $K^\times\to K^\times/K^{\times 2}$ and the discriminant map $\GW(K)\to K^\times/K^{\times 2}$. There is a well-defined injection
	\begin{equation}\label{eq:harderinj}
	\pi_0^{\naiv}\mathscr{S}_n(K)\to\GW(K)\times_{K^\times/K^{\times2}} K^\times
	\end{equation}
	that takes the naive connected component of a symmetric matrix to the pair of its class in $\GW(K)$ and its determinant.
\end{thm}

\begin{proof}
	See for example \cite[Proposition~3.9]{caza}.
\end{proof}

Recall that the writhe is constructed as follows: let $\alpha$ be a theta characteristic of $C$ with $h^0(C,\alpha)=0$ and let $\psi\colon C\hookrightarrow \pp^3$ be an embedding such that $L=\psi^*\cO_{\pp^3}(1)$ is $3$-very ample. The corresponding Ulrich sheaf $\mathcal{F}_{\alpha}$ of \Cref{thm:writheulrich} provides a symmetric matrix $\Lambda$ of linear forms on $\wedge^4 V$. This defines a morphism $\mathcal{W}\to\mathscr{S}_N$ where $N$ is the degree of the secant variety. Taking the naive connected component, and using \Cref{eq:harderinj} this gives us a map
\begin{equation}\label{eq:l}
\w(-,\alpha): \pi_0^{\naiv}\mathcal{W}_{L,3}(K)\to\GW(K)\times_{K^\times/K^{\times2}} (K^\times).
\end{equation}
By \Cref{cor:scalingpower4} this induces a map
\begin{equation}\label{eq:writheinv}
\w(-,\alpha)\colon [C,\pp^3]^{\iso}_L\to\GW(K)\times_{K^\times/K^{\times2}} (K^\times/K^{\times4N}).
\end{equation}

\begin{Def}
    In the above setting, we define the \emph{arithmetic writhe} of the embedding $\psi$ to be the image $\w([\psi],\alpha)$ of the isotopy class of $\psi$ under the map from \Cref{eq:writheinv}.
\end{Def}

\begin{rem}\label{rem:writheemb}
	The arithmetic writhe of the embedding is automatically invariant under algebraic isotopies. Furthermore, The first component of $\w([\psi],\alpha)$ is the arithmetic writhe of the embedded curve $\psi(C)$. In particular, this shows that if $\psi,\psi'\colon C\hookrightarrow \pp^3$ are two algebraically isotopic embeddings, then the writhes of $\psi(C)$ and $\psi'(C)$ are the same. 
    
    We also notice that, if $K$ is an algebraically or real closed field, then we have $K^{\times4N}=K^{\times2}$ which implies that the second factor of $\w([\psi],\alpha)$ is superfluous: the arithmetic writhe of  the embedding $\psi$ coincides with the arithmetic writhe of the embedded curve $\psi(C)$.
\end{rem}

\subsection{Rational curves of degree four}
In this subsection we consider the case $C=\pp^1$. As $\alpha=\cO_{\pp^1}(-1)$ is the only theta characteristic of $\pp^1$, we denote $\w(-)=\w(-,\alpha)$. For $n\geq3$ we write $[\pp^1,\pp^3]^{\iso}_{n}=[\pp^1,\pp^3]^{\iso}_{\cO_{\pp^1}(n)}$. Our goal is to prove that 
\begin{equation*}
	\w\colon [\pp^1,\pp^3]^{\iso}_4\to\GW(K)\times_{K^\times/K^{\times2}} (K^\times/K^{\times12})
\end{equation*}
is injective in this case. In other words, the isotopy type of two embeddings of $\pp^1$ to $\pp^3$ of degree four is completely determined by the arithmetic writhe. To this end, we will prove that $\pi_0^{\naiv}\mathcal{W}_{\mathcal{O}_{\pp^1}(4),3}(K)$ is in bijection to the set of pointed naive homotopy classes of morphisms $\pp^1\to \pp^1$ of degree three. These were characterized in \cite{caza}. We first have to introduce some notation. Following \cite[Definition~2.1]{caza} we consider the scheme $\mathscr{T}_n$, $n\in\N$, of pointed rational functions on $\pp^1$, realized as the open subscheme of the affine space  
\begin{equation*}
	\A^{2n}=\Spec K[a_0,\ldots,a_{n-1},b_0,\ldots,b_{n-1}]
\end{equation*}
where the resultant of the two polynomials $t^n+a_{n-1}t^{n-1}+\cdots+a_0$ and $b_{n-1}t^{n-1}+\cdots+b_0$ does not vanish. Hence $\mathscr{T}_n(K)$ is the set of all pairs $(f,g)$ of coprime polynomials $f,g\in K[t]$ where $f$ is monic of degree $n$ and $g$ has degree at most $n-1$. For such a pair one can write
\begin{equation*}
	\frac{g}{f}=h_0(f,g)\cdot t^{-1}+h_1(f,g)\cdot t^{-2}+h_2(f,g)\cdot t^{-3}+\cdots
\end{equation*}
for suitable $h_i(f,g)\in K$ and we define the \emph{Hankel matrix} of $f$ and $g$ as 
\begin{equation*}
	\Hank_n(f,g)=(h_{i+j}(f,g))_{0\leq i,j\leq n-1}.
\end{equation*}
A related matrix is the \emph{B\'ezout matrix} which is constructed as follows. The polynomial $f(x)g(y)-f(y)g(x)$ is divisible by $x-y$ so that we can write
\begin{equation*}
	\frac{f(x)g(y)-f(y)g(x)}{x-y}=\sum_{1\leq i,j\leq n} c_{ij} x^{i-1}y^{j-1}
\end{equation*}
for $c_{ij}\in K$. Then the B\'ezout matrix is defined as
\begin{equation*}
	\B_n(f,g)=(c_{ij})_{1\leq i,j\leq n}.
\end{equation*}

It is classical knowledge that $\Hank_n(f,g)$ is equivalent as symmetric bilinear form to the B\'ezout matrix $\B_{n}(f, g)$ and these two matrices have the same determinant, see for example \cite[Proposition 1.9.7]{era} for a proof that, albeit formulated over real closed fields only, works over every field. In particular, the determinant of $\Hank_n(f,g)$ is non-zero if and only if $f$ and $g$ are co-prime. 

\begin{thm}\label{thm:translation}
	The morphism $h\colon \mathscr{T}_3\to \mathcal{W}_{\mathcal{O}_{\pp^1}(4),3}$ that sends a pair $(f,g)$ to
	\begin{equation*}
		\varphi(f,g)=h_4(f,g)\cdot (t\wedge t^2\wedge t^3\wedge t^4)+\cdots+h_0(f,g)\cdot (1\wedge t\wedge t^2\wedge t^3)
	\end{equation*}
	is well-defined and surjective on $K$-points. It satisfies $\Hank_3=\Lambda\circ h$ where $\Lambda$ is the matrix from \Cref{thm:writheulrich}.
\end{thm}

\begin{proof}
	The property $\Hank_3=\Lambda\circ h$ follows from the explicit expression of $\Lambda$ is \Cref{eq:hankel}. This also implies that $\varphi$ is well-defined because the determinant of $\Hank_3$ does not vanish on $\mathscr{T}_3$. For the surjectivity statement it suffices to show that every $h_0,h_1,h_2,h_3\in K$ whose Hankel matrix has rank three can be completed to the sequence of coefficients of the power series expansion of a suitable rational function. This is an easy linear algebra exercise.
\end{proof}

\begin{cor}\label{cor:rat4inj}
	The arithmetic writhe 
	\begin{equation*}
		\w\colon [\pp^1,\pp^3]^{\iso}_4\to\GW(K)\times_{K^\times/K^{\times2}} (K^\times/K^{\times12})
	\end{equation*}
	is injective.
\end{cor}

\begin{proof}
	We consider the composition of maps
	\begin{equation*}
		\pi_0^{\naiv}\mathscr{T}_3(K)\to\pi_0^{\naiv}\mathcal{W}_{\mathcal{O}_{\pp^1}(4),3}(K)\to\GW(K)\times_{K^\times/K^{\times 2}} K^\times.
	\end{equation*}
	This is the map for which it was proven in \cite[\S3.3]{caza} that it is injective. Indeed, this follows from $\Hank_3=\Lambda\circ h$ and from the fact that $\Hank_3(f,g)$ and $\B_{3}(f,g)$ have the same determinant and the same class in $\GW(K)$.
	By \Cref{thm:translation} the first map is surjective which implies that the second map is injective. Now the claim follows from \Cref{cor:scalingpower4}.
\end{proof}

\begin{rem}
	By the same argument as in \Cref{prop:rational4} we can deduce that the image of the map of \Cref{cor:rat4inj} consists of all pairs whose first component can be represented as $\sum_{i=1}^3\langle a_i\rangle$ for $a_i\in K^\times$.
\end{rem}

\begin{rem}\label{rem:realrationalcurves}
	In the case $K=\R$, \Cref{rem:writheemb} shows the arithmetic writhe of the embedding coincides with the arithmetic writhe of the embedded curve. Thus, we can see it as a map
	\begin{equation*}
		\w\colon [\pp^1,\pp^3]^{\iso}_n\to\GW(\R)
	\end{equation*}
    As the rank of $\w([\psi])$ is determined by $n$, we do not loose information when post-composing with the signature.
	We can thus view $\w$ as a map
	\begin{equation*}
		\w\colon[\pp^1,\pp^3]^{\iso}_n\to\Z
	\end{equation*}
	which is the encomplexed writhe introduced by Viro in \cite{virowrithe}. In \cite{virowrithe}, it was shown that $\w$ is an invariant under the possibly coarser equivalence relation of \emph{rigid isotopy}. In other words, the map $\w$ factors through the set $[\pp^1,\pp^3]^{\rm rig}_n$ of rigid isotopy classes:
	\begin{equation*}
		[\pp^1,\pp^3]^{\iso}_n\to[\pp^1,\pp^3]^{\rm rig}_n \to\Z.
	\end{equation*}
	It was shown in \cite[Remark~3.21]{bjoerk} that the map $[\pp^1,\pp^3]^{\rm rig}_6 \to\Z$ is not injective. This implies that $\w\colon [\pp^1,\pp^3]^{\iso}_6\to\GW(\R)$ is also not injective. On the other hand, by \cite[Theorem~1.2]{bjoerk} the map $[\pp^1,\pp^3]^{\rm rig}_5\to\Z$ is injective. We do not know whether this generalizes to our setup, i.e., whether $\w\colon [\pp^1,\pp^3]^{\iso}_5\to\GW(\R)$ is injective or maybe even $\w\colon [\pp^1,\pp^3]^{\iso}_5\to\GW(K)\times_{K^\times/K^{\times2}} (K^\times/K^{\times4N})$ is injective for every $K$.
\end{rem}

\begin{ex}\label{ex:newinvariants}
 The \texttt{Macaulay2} code ``invariants.m2'', attached as ancillary file to this arXiv submission, computes two Ulrich sheaves on the secant variety $\Sigma \subset \pp^6$ of a rational normal curve of degree $6$: the first one is the Ulrich sheaf $\cF$ of rank one coming from the theta characteristic $\alpha= \mathcal{O}_{\pp^1}(-1)$ as in \Cref{thm:ulrichsecant} and second is the  sheaf $\cF_6$ of rank five coming from \ref{thm:ulrichsecant}. Then it computes for each of those a symmetric matrix of linear forms  as in Theorem \ref{thm:ulrichdegreewell}. Note that, although we have not shown that the sheaf $\cF_6$ is a \emph{symmetric} Ulrich sheaf, it turns out that the resulting matrix is symmetric. We denote these matrices by $\Lambda$ and $\Lambda'$ respectively. The signature of $\Lambda$ evaluated at a linear subspace $E$ of dimension two in $\pp^6$ that is disjoint from the secant variety is the writhe of the projected curve in $\pp^3$ by \Cref{thm:writheulrich}. The signature of $\Lambda'$ at such a space is a new invariant of the curve which also does not change along algebraic or rigid isotopies. Our code computes these two invariants for three different linear spaces $E$. Here we record the results:
 \begin{enumerate}
  \item The linear space $E$ is the row-span of the matrix
  \[\begin{pmatrix} 2 &  9 &  3 &  1 &  8 &  7 &  3 \\ 1 &  6 &  5 &  1 &  9 &  6 &  8 \\ 7 &  8 &  1 &  1 &  2 &  6 &  9 \end{pmatrix}.\]
  Both invariants of the projected curve in $\pp^3$ are equal to $0$.
  \item The linear space $E$ is the row-span of the matrix
  \[\begin{pmatrix} 2 &  1 &  7 &  2 &  4 &  2 &  0 \\ 7 &  6 &  9 &  0 &  3 &  6 &  8 \\ 9 &  6 &  8 &  3 &  5 &  1 &  7 \end{pmatrix}.\]
  The writhe of the projected curve in $\pp^3$ is $0$ while the new invariant is $4$. 
  \item The linear space $E$ is the row-span of the matrix
  \[\begin{pmatrix} 0! & 1! & 2! & 3! & 4! & 5! & 6! \\ 1! & 2! & 3! & 4! & 5! & 6! & 7! \\ 2! & 3! & 4! & 5! & 6! & 7! & 8! \end{pmatrix}.\]
  The writhe of the projected curve in $\pp^3$ is $10$. Note that this is predicted by \cite[Corollary 4.15]{kummersinn}. The new invariant is $-2$.
 \end{enumerate}
 In particular, these computations show that the new invariant can distinguish non-isotopic curves which cannot be distinguished by the writhe. It would be interesting to have a local description for this new invariant similar to the one of the writhe.
\end{ex}

\bigskip

 \noindent \textbf{Acknowledgments.}
We would like to thank the organizers of the Thematic Einstein Semester on ``Algebraic Geometry: Varieties,
Polyhedra, Computation'' in Berlin during which this project was initiated. We further thank Claudiu Raicu, Stephen McKean, Jinhyung Park and Kirsten Wickelgren for their helpful comments.

 \bibliographystyle{alpha}
 \bibliography{biblio}
 \end{document}